\documentclass[11pt,reqno]{amsart}

\usepackage{amsmath, amsfonts, amsthm, amssymb, color}

\newtheorem{Theorem}{Theorem}[section]
\newtheorem{Lemma}{Lemma}[section]

\theoremstyle{definition}
\newtheorem{Definition}{Definition}[section]

\theoremstyle{remark}
\newtheorem{Remark}{Remark}[section]

\numberwithin{equation}{section}

\newcommand{\eps}{\varepsilon}
\newcommand{\Cin}{C^\infty_c(\R^+\times\O)}

\newcommand{\sro}{{\sqrt{\rho}}}

\renewcommand{\v}{{ v}}

\newcommand{\vfi}{{\varphi}}

\newcommand{\Tn}{{\mathbb T}_\mu}

\newcommand{\Sk}{{\mathbb S}_r}
\newcommand{\Sn}{{\mathbb S}_\mu}

\renewcommand{\u}{{ u}}

\newcommand{\R}{{\mathbb R}}
\newcommand{\Dv}{{\rm div}}

\newcommand{\m}{{ m}}

\def\w{w}

\newcommand{\T}{{\mathcal T}}
\def\f{\frac}
\renewcommand{\O}{\Omega}

\def\D{\Delta }

\def\hf1{^\f{1}{1-\xi^2}}

\def\be{\begin{equation}}
\def\en{\end{equation}}
\def\bs{\begin{split}}
\def\es{\end{split}}

\renewcommand{\v}{{ v}}

\author{Didier Bresch}
\address{LAMA UMR5127 CNRS, Universit\'e Savoie Mont-Blanc, France}
\email{didier.bresch@univ-smb.fr}

\author{Alexis F. Vasseur}
\address{Department of Mathematics,
The University of Texas at Austin.}
\email{vasseur@math.utexas.edu}
\author{Cheng Yu}
\address{Department of Mathematics,
University of Florida.}
\email{chengyu@ufl.edu}

\title[Existence of the Compressible Navier-Stokes Equations]
{Global Existence of Entropy-Weak Solutions to the 
 Compressible Navier-Stokes Equations with Non-Linear Density Dependent Viscosities}
\date{\today}

\begin{document}

\begin{abstract}
In this paper, we extend considerably the global existence results of entropy-weak solutions
related to compressible Navier-Stokes system  with density dependent viscosities obtained,
independently (using different strategies), by Vasseur-Yu [{\it Inventiones mathematicae} 
(2016) and arXiv:1501.06803 (2015)] and by Li-Xin [arXiv:1504.06826 (2015)]. 
 More precisely we are able to consider a physical symmetric viscous stress tensor $\sigma = 2 \mu(\rho) \,{\mathbb D}(\u) + \bigl(\lambda(\rho) {\rm div} \u - P(\rho)\bigr) \, {\rm Id}$ 
where ${\mathbb D}(\u) = [\nabla \u + \nabla^T \u]/2$ with a shear and  bulk viscosities (respectively $\mu(\rho)$ and $\lambda(\rho)$) satisfying the BD relation $\lambda(\rho)=2(\mu'(\rho)\rho - \mu(\rho))$ and  a pressure law $P(\rho)=a\rho^\gamma$ (with $a>0$ a given constant)  for any adiabatic constant $\gamma>1$.  The nonlinear shear viscosity $\mu(\rho)$ satisfies some lower and upper bounds for low and high densities (our mathematical result includes the case $\mu(\rho)= \mu\rho^\alpha$ with $2/3 < \alpha < 4$ and $\mu>0$ constant).
    This provides an answer to a longstanding mathematical question on compressible Navier-Stokes
equations with density dependent viscosities as  mentioned for instance by F. Rousset in the
Bourbaki 69ème année, 2016--2017, no 1135.
\end{abstract}

\maketitle

\section{Introduction}  When a fluid is governed by the barotropic compressible Navier-Stokes equations, the existence of global weak solutions, in the sense of J. {\sc Leray} (see \cite{Le}), in space dimension greater than two remained for a long time  without answer, because of the weak control of the divergence of the velocity field which may provide the possibility for the density to vanish (vacuum state) even if initially this is not the case. 

There exists a huge literature on this question, in the case of constant  shear viscosity $\mu$ and constant bulk viscosity $\lambda$.  Before 1993, many authors such as  Hoff \cite{Hoff87},  Jiang-Zhang \cite{JZ},  Kazhikhov--Shelukhin \cite{KS}, Serre \cite{S}, Veigant--Kazhikhov \cite{VK}  (to cite just some of them) have obtained partial answers: We can cite, for instance, the  works in dimension 1 in 1986 by  Serre \cite{S}, the one  by Hoff \cite{Hoff87} in 1987, and the one in the spherical case in 2001 by Jiang-Zhang \cite{JZ}. The first rigorous approach of this problem in its generality is due in 1993 by P.--L. Lions \cite{Lions} when the pressure law in terms of the density is given by  $P(\rho)=a \rho^\gamma$ where $a$ and $\gamma$ are two  strictly positive constants. He has presented in 1998 a complete theory for  $P(\rho)=a \rho^\gamma$ with $\gamma\ge 3d/(d+2)$ (where $d$ is the space dimension) allowing to obtain the result of global existence of weak solutions \`a la Leray  in dimension $d=2$ and $3$ and for general initial data belonging to the energy space.  His result has been then extended in 2001 to the case $P(\rho)= a \rho^\gamma$ with $\gamma>d/2$ by Feireisl-Novotny-Petzeltova \cite{FNP} introducing an appropriated method of truncation. Note also in 2014 the paper by  Plotnikov-Weigant \cite{PW} in dimension 2 for the linear pressure law that means $\gamma =1$. In 2002,  Feireisl \cite{F04} has also proved it is possible to consider a  pressure $P(\rho)$ law non-monotone on a compact set $[0,\rho_*]$ (with $\rho_*$ constant) and monotone elsewhere. This has been relaxed  in 2018 by  Bresch-Jabin \cite{BJ} allowing to consider real non-monotone  pressure laws. They have also proved that it is possible to consider some constant anisotropic viscosities.  The Lions theory has also been extended recently by Vasseur-Wen-Yu \cite{VWY} to pressure laws depending on two phases (see also Mastese $\&$ {\it al.} \cite{MaMiMuNoPoZa},  Novotny \cite{No} and Novotny-Pokorny \cite{NoPo}). The method introduced by Bresch-Jabin in \cite{BJ}  has also been recently developped in the bifluid framework by Bresch-Mucha-Zatorska in \cite{BrMuZa}.

When the shear and the bulk viscosities (respectively $\mu$ and $\lambda$) are assumed to depend on the density $\rho$, the mathematical framework is completely different. It has been discussed, mathematically, initially in a paper by  Bernardi-Pironneau \cite{BP} related to viscous shallow-water equations and by P.--L. Lions \cite{Lions} in his second volume
related to mathematics and fluid mechanics.  The main ingredient in the constant case which is the compactness in space of the effective flux  $F= (2\mu+\lambda) {\rm div} u - P(\rho)$ is no longer true for density dependent viscosities.   In space dimension greater than one, a real breakthrough has been realized with a series of papers by Bresch-Desjardins \cite{BD,BD2006,BrDeFormula, BrDeSpringer}, (started in 2003 with Lin \cite{BDL} in the  context of Navier-Stokes-Korteweg with linear shear viscosity case) who have identified an information related to the gradient of a function of the density if the viscosities satisfy what is called the Bresch-Desjardins constraint.  This information is usually called the BD entropy in the literature with the introduction of the concept of entropy-weak solutions. Using such extra information, they obtained the global existence of entropy-weak solutions in the presence of appropriate drag terms or singular pressure close to vacuum. Concerning the one-dimensional in space case or the spherical case, many important results have been obtained  for instance by  Burtea-Haspot \cite{BuHa},  Ducomet-Necasova-Vasseur \cite{DNV},  Constantin-Drivas-Nguyen-Pasqualottos \cite{CoDrNgPa}, Guo-Jiu-Xin \cite{GJX},  Haspot \cite{Haspot}, Jiang-Xin-Zhang \cite{JXZ}, Jiang-Zhang \cite{JZ},  Kanel \cite{Kan},  Li-Li-Xin \cite{LiLiXi},  Mellet-Vasseur~\cite{MV2},  Shelukhin \cite{S}  without such kind of additional terms. Stability and construction of approximate solutions in space dimension two or three have been investigated during more than fifteen years with a first important stability result without drag terms or singular pressure by  Mellet-Vasseur \cite{MV}.  Several important works for instance by Bresch-Desjardins \cite{BD,BD2006,BrDeFormula, BrDeSpringer} and Bresch-Desjardins-Lin \cite{BDL}, Bresch-Desjardins-Zatorska \cite{BDZ}, Li-Xin \cite{LiXi}, Mellet-Vasseur \cite{MV}, Mucha-Pokorny-Zatorska
\cite{MuPoZa},  Vasseur-Yu \cite{VY-1,VY},  and  Zatorska \cite{Z} have also been written trying to find a way to construct approximate solutions.  Recently a real breakthrough has been done in  two important papers by Li-Xin \cite{LiXi} and Vasseur-Yu \cite{VY}:  Using two different ways, they got the global existence of entropy-weak solutions for the compressible paper when  $\mu(\rho)=\rho$ 
and $\lambda(\rho)=0$. Note that in the last paper \cite{LiXi} by Li-Xin, they also consider more general viscosities satisfying the BD relation but with a non-symmetric stress diffusion ($\sigma = 
\mu(\rho)\nabla u + (\lambda(\rho){\rm div} u - P(\rho)) {\rm Id}$) and more restrictive conditions on the shear $\mu(\rho)$ viscosity and bulk viscosity  $\lambda(\rho)$ and on the pressure law $P(\rho)$ compared to the present paper.

    The objective of this current paper is to extend the  existence results of global entropy-weak solutions obtained independently (using different strategies) by Vasseur-Yu \cite{VY} and Lin-Xin \cite{LiXi} to answer a longstanding mathematical question on compressible Navier-Stokes equations with density dependent viscosities as  mentioned for instance by Rousset \cite{Ro}.  More precisely extending and coupling carefully the two-velocities framework by Bresch-Desjardins-Zatorska \cite{BDZ} with the generalization of the quantum B\"ohm identity found by Bresch-Couderc-Noble-Vila \cite{BCNV} (proving a generalization of the dissipation inequality used  by J\"ungel \cite{J}  for Navier-Stokes-Quantum system and established by J\"ungel-Matthes in \cite{JuMa}) and with the renormalized solutions introduced in Lacroix-Violet and Vasseur \cite{LaVa}, we can get global existence of entropy-weak solutions to the following Navier-Stokes equations:
\begin{equation}
\label{NS equation}
\begin{split}
&\rho_t+\Dv(\rho\u)=0\\
&(\rho\u)_t+\Dv(\rho\u\otimes\u)+\nabla P(\rho) - 2 {\rm div}\bigl(\sqrt{\mu(\rho)} \mathbb{S}_\mu
     + \frac{\lambda(\rho)}{2\mu(\rho)} {\rm Tr}(\sqrt{\mu(\rho)} \mathbb S_\mu) {\rm Id} \bigr)=0,
\end{split}
\end{equation}
where
$$\sqrt{\mu(\rho)} \mathbb{S}_\mu = 
     \mu(\rho) {\mathbb D}(\u) $$
with  data
\begin{equation}
\label{initial data}
\rho|_{t=0}=\rho_0(x)\ge 0,\;\;\;\;\;\rho\u|_{t=0}=\m_0(x)=\rho_0\u_0,
\end{equation}
and where $P(\rho) =a \rho^{\gamma}$ denotes the pressure with the two constants $a>0$ 
and $\gamma >1$, $\rho$ is the density of fluid, $\u$ stands for the velocity of fluid, $\mathbb{D}\u=[\nabla\u+\nabla^T\u]/2$ is the strain tensor. 
As usually, we consider
$$\u_0= \frac{m_0}{\rho_0} \hbox{ when } \rho_0\not=0 \hbox{ and }\u_0 = 0 \hbox{ elsewhere},
\qquad \frac{|m_0|^2}{\rho_0} = 0  \hbox{  a.e. on }  \{x\in \Omega: \rho_0(x) = 0\}. $$
 We remark the following identity \begin{equation*}
 2 {\rm div}\bigl(\sqrt{\mu(\rho)} \mathbb{S}_\mu
     + \frac{\lambda(\rho)}{2\mu(\rho)} {\rm Tr}(\sqrt{\mu(\rho)} \mathbb S_\mu) {\rm Id} \bigr)=-2\Dv(\mu(\rho)\mathbb{D}\u)-\nabla(\lambda(\rho)\Dv\u).
\end{equation*}

 The viscosity coefficients $\mu=\mu(\rho)$ and $\lambda=\lambda(\rho)$ satisfy the Bresch-Desjardins relation introduced in \cite{BrDeFormula}
\begin{equation}
\label{BD relationship}
\lambda(\rho)=2(\rho\mu'(\rho)-\mu(\rho)).
\end{equation}
The relation between the stress tensor $\mathbb{S}_\mu$ and the triple 
$(\mu(\rho)/\sqrt\rho, \sqrt \rho \u, \sqrt\rho \v)$ where $\v= 2 \nabla s(\rho)$ with $s'(\rho)= \mu'(\rho)/\rho$  will be proved in the following way: The matrix
$\mathbb{S}_\mu$ is the symetric part of a matrix value function $\mathbb{T}_\mu$ namely
\begin{equation}\label{Smu}
\mathbb{S}_\mu = \frac{(\mathbb{T}_\mu + \mathbb{T}_\mu^t)}{2}
\end{equation}
where $\mathbb{T}_\mu$ is defined through
\begin{equation} \label{Tmu}
\begin{split}
\sqrt{\mu(\rho)} \mathbb{T}_\mu 
= \nabla (\sqrt\rho \u\,  \frac{\mu(\rho)}{\sqrt\rho}) 
    - \sqrt\rho \u \otimes  \sqrt\rho \nabla s(\rho) 
\end{split}
\end{equation}
with 
\begin{equation}\label{s} 
s'(\rho) = \mu'(\rho) /\rho,
\end{equation}
and
\begin{equation} \label{Tmu1}
\begin{split}
 \frac{\lambda(\rho)}{2\mu(\rho)} {\rm Tr}(\sqrt{\mu(\rho)} \mathbb T_\mu) {\rm Id} 
= \Bigl[ {\rm div}(\frac{\lambda(\rho)}{\mu(\rho)}  \sqrt\rho \u\,  \frac{\mu(\rho)}{\sqrt\rho}) 
     - \sqrt\rho \u \cdot  \sqrt\rho \,\nabla s(\rho) \, \frac{\rho \mu''(\rho)}{\mu'(\rho)}\Bigr] {\rm Id}.
\end{split}
\end{equation}
  For the sake of simplicity, we will consider the case of periodic boundary conditions in three dimension in space namely $\O=\mathbb{T}^3$. In the whole paper, we assume:
 \begin{equation} \label{regmu}
  \mu  \in C^0({\mathbb R_+}; \, {\mathbb R_+})\cap C^2({\mathbb R}_+^*; \,{\mathbb R}),
\end{equation} 
 where $\mathbb R_+=[0,\infty) \text{ and } \mathbb R_+^*=(0,\infty).$
 We also assume that  there exists two positive numbers $\alpha_1,\alpha_2$ such that 
  \begin{equation}
  \label{mu estimate}
  \begin{array}{l}
 \displaystyle{
   \frac{2}{3}<\alpha_1<\alpha_2<4,
   }\\[0.3cm]
 \displaystyle{\mathrm{for \ any } \ \rho>0, \qquad 
0<\frac{1}{\alpha_2}\rho \mu'(\rho)\leq \mu(\rho)\leq  \frac{1}{\alpha_1}\rho \mu'(\rho),
}
\end{array}
  \end{equation}
and there exists a constant $C>0$ such that 
\begin{equation}  \label{mu estimate1}
\left|\frac{\rho \mu''(\rho)}{\mu'(\rho)}\right| \le C < +\infty.
\end{equation}
  Note that if $\mu(\rho)$ and $\lambda(\rho)$ satisfying \eqref{BD relationship} and \eqref{mu estimate}, then
  $$\lambda(\rho) + 2\mu(\rho)/3 \ge 0$$
and thanks to \eqref{mu estimate} 
$$ \mu(0)=  \lambda(0) = 0.$$ 
Note that the hypothesis \eqref{mu estimate}--\eqref{mu estimate1} allow a shear viscosity of the form $\mu(\rho)=\mu \rho^{\alpha}$ with $\mu>0$ a constant where $2/3<\alpha<4$ and a bulk viscosity satisfying the BD relation: $\lambda(\rho)= 2(\mu'(\rho)\rho - \mu(\rho))$. 
 
 \medskip
 
\noindent {\bf Remark.}  
  In \cite{VY} and \cite{LiXi} the case $\mu(\rho)=\mu\rho$ and $\lambda(\rho)=0$ is considered, and in \cite{LiXi} more general cases have been considered but  with a non-symmetric viscous term in the three-dimensional in space case, namely $- \Dv (\mu(\rho)\nabla \u) - \nabla (\lambda(\rho)\Dv \u)$. In \cite{LiXi} the viscosities $\mu(\rho)$ and $\lambda(\rho)$ satisfy \eqref{BD relationship} with
  $\mu(\rho) = \mu \rho^\alpha$ where  $\alpha \in [3/4,2)$ and with the following assumption on the value $\gamma$ for the pressure $p(\rho)=a\rho^\gamma$:  
 $$\hbox{ If } \alpha\in [3/4,1], \qquad \gamma \in (1,6\alpha-3)$$
 and 
 $$\hbox{ if } \alpha \in (1,2), \qquad \gamma\in [2\alpha-1,3\alpha-1].$$ 

\bigskip

\noindent The main result of our paper reads as follows:
 \begin{Theorem}
\label{main result}
Let $\mu(\rho)$ verify \eqref{regmu}--\eqref{mu estimate1} and   $\mu$ and $\lambda$ verify \eqref{BD relationship}.  Let us assume the initial data satisfy
\begin{equation}
\label{initial energy}
\begin{split}
& \int_{\O}\left(\frac{1}{2}\rho_0|\u_0+ 2\kappa \nabla s(\rho_0)|^2 
   +\kappa(1-\kappa)\rho_0\frac{|2\nabla s(\rho_0)|^2}{2}\right) \, dx \\
& \hskip6cm 
+ \int_{\O}\left(a\frac{\rho_0^{\gamma}}{\gamma-1} + \mu(\rho_0)\right)\,dx\leq C <+\infty.
\end{split}
\end{equation}
with $k\in (0,1)$ given. Let $T$ be given such that $0<T<+\infty$, then, for any $\gamma>1$, there exist a renormalized solution to \eqref{NS equation}-\eqref{initial data} as defined in Definition \ref{def_renormalise_u}. Moreover, 
this renormalized solution 
with initial data satisfying \eqref{initial energy} is a weak solution to \eqref{NS equation}-\eqref{initial data} in the sense of Definition~\ref{defweak}.

\end{Theorem}
Our result may be considered as an improvement of \cite{LiXi} for two reasons: First it takes into account a physical symmetric viscous tensor and secondly, it extends the range of coefficients $\alpha$ and $\gamma$. 
The method is based on the consideration of an approximated system with an extra pressure quantity, appropriate non-linear drag terms and appropriate capillarity terms. This generalizes the Quantum-Navier-Stokes system with quadratic drag terms considered in \cite{VY-1,VY}. 
First we prove that weak solutions of the approximate solution are renormalized solutions of the system, in the sense of \cite{LaVa}. Then we pass to the limit with respect to $r_2,r_1, r_0, r, \delta$ to get renormalized solutions of the compressible Navier-Stokes system.   The final step concerns the proof that a renormalized solution of the compressible Navier-Stokes system is a global weak solution of the compressible Navier--Stokes system.
Note that, thanks to the technique of renormalized solution introduced in \cite{LaVa}, it is not necessary to derive the Mellet-Vasseur type inequality in this paper: This allows us to  cover the all range $\gamma>1$. 

\medskip

\noindent {\it First Step.}
   Motivated by the work of \cite{LaVa}, the first step is to establish the existence of global $\kappa$ entropy weak solution to the following approximation
   \begin{equation}
\label{last level approximation}
\begin{split}
&\rho_t+\Dv(\rho\u)=0\\
&(\rho\u)_t+\Dv(\rho\u\otimes\u)+\nabla P(\rho) + \nabla P_\delta(\rho)   \\ 
 &\hskip3cm- 2 {\rm div}\Bigl(\sqrt{\mu(\rho)} \mathbb{S}_\mu 
    + \frac{\lambda(\rho)}{2\mu(\rho)} {\rm Tr}(\sqrt{\mu(\rho)} \mathbb S_\mu) {\rm Id}\Bigr) \\
& \hskip3cm-  2 r {\rm div}\Bigl(\sqrt{\mu(\rho)}  \mathbb{S}_r
    + \frac{\lambda(\rho)}{2\mu(\rho)} {\rm Tr}(\sqrt{\mu(\rho)} \mathbb S_r) {\rm Id}\Bigr) \\
& \hskip7cm + r_0\u+r_1\frac{\rho}{\mu'(\rho)}|\u|^2\u+r_2\rho|\u|\u = 0
\end{split}
\end{equation}
where the barotorpic pressure law and the extra pressure term are respectively
 \begin{equation}
 P(\rho)= a\rho^\gamma, \qquad   P_\delta (\rho)= \delta \rho^{10} \hbox{ with } \delta>0.
 \end{equation}
 The matrix $\mathbb{S}_\mu$ is defined in \eqref{Smu} and $\mathbb{T}_\mu$ is given in\eqref{Tmu}- \eqref{Tmu1}.
The matrix  $\mathbb{S}_r$ is compatible in the following sense:
\begin{equation}\label{eq_quantic}
\begin{split}
r\sqrt{\mu(\rho)}  \mathbb{S}_r   = 2r \Bigl[2 \sqrt{\mu(\rho)} \nabla\nabla Z(\rho) 
    - \nabla (\sqrt{\mu(\rho)} \nabla Z(\rho))\Bigr],
\end{split}
\end{equation}
where 
\begin{equation} \label{ZZ}
\displaystyle Z(\rho) = \int_0^\rho [(\mu(s))^{1/2} \mu'(s)]/s \, ds, \qquad 
\displaystyle k(\rho) = \int_0^\rho [{\lambda(s)\mu'(s)}]/{\mu(s)^{3/2}} ds
\end{equation}
and
\begin{equation}\label{eq_quantic11}
\begin{split}
 r\frac{\lambda(\rho)}{2\mu(\rho)} {\rm Tr}(\sqrt{\mu(\rho)} \mathbb S_r) {\rm Id} 
=   r(\frac{\lambda(\rho)}{\sqrt{\mu(\rho)}}+ \frac{1}{2} k(\rho))\Delta Z(\rho) {\rm Id}
      - \frac{r}{2}{\rm div} [ k(\rho)\nabla Z(\rho)] {\rm Id}.
\end{split}
\end{equation}
 

\noindent {\bf Remark.} Note that the previous system is the generalization of the quantum viscous Navier-Stokes system considered by Lacroix-Violet and Vasseur in \cite{LaVa} (see also the interesting papers by Antonelli-Spirito \cite{AnSp1, AnSp2} and by Carles-Carrapatoso-Hillairet \cite{CaCaHi}). Indeed if we consider  $\mu(\rho)=\rho$ and $\lambda(\rho)=0$, we can write $\sqrt{\mu(\rho)} \mathbb S_r$ as
$$\sqrt{\mu(\rho)} \mathbb{S}_r =4 \sqrt{\rho} \Bigl[ \nabla\nabla \sqrt\rho
      -  4 (\nabla \rho^{1/4} \otimes \nabla \rho^{1/4}) \Bigr],
$$
using  $Z(\rho) = 2\sqrt\rho.$ The Navier--Stokes equations for quantum fluids was also considered by A. J\" ungel in \cite{J}.

\bigskip

\noindent As the first step generalizing \cite{VY},  we  prove the following result.
\begin{Theorem}
\label{main result 1} Let $\mu(\rho)$ verifies \eqref{regmu}--\eqref{mu estimate1} and $\lambda(\rho)$ is given by \eqref{BD relationship}. If $r_0>0$, then we assume also that ${\rm inf}_{s \in [0,+\infty)} \mu'(s)=\epsilon_1 >0$. 
   Assume that  $r_1$ is small enough compared to $r$,  $r_2$  is small enough compared to $\delta$, and that the initial values verify 
   \begin{equation}\label{Initial conditions}
   \begin{split}
  & \int_\Omega \rho_0\left(\frac{|\u_0+2\kappa\nabla s(\rho_0)|^2}{2}+(\kappa (1-\kappa)+r)\frac{|2\nabla s(\rho_0)|^2}{2}\right) \, dx\\
   & \hskip4cm + \int_\Omega \bigl(a \frac{\rho_0^\gamma}{\gamma-1}+
    \mu(\rho_0) + \delta \frac{\rho_0^{10}}{9}+\frac{r_0}{\varepsilon_1}|(\ln \rho_0)_-|\bigr)\,dx < + \infty,
   \end{split}
   \end{equation} 
for a fixed $\kappa\in (0,1)$.  Then there exists a $\kappa$ entropy weak solution $(\rho,\u, \mathbb T_\mu, \mathbb S_r)$ to
   \eqref{last level approximation}--\eqref{eq_quantic11} satisfying the initial conditions \eqref{initial data}, in the sense that
$(\rho,\u, \mathbb T_\mu, \mathbb S_r)$ satisfies the mass and momentum equations in a weak form,  and satisfies the compatibility formula  in the sense of definition \ref{defweak}. In addition, it verifies  the following estimates:
\begin{equation}
\label{priori estimates} 
\begin{split}
&\|\sqrt{\rho}\, (\u+2\kappa \nabla s(\rho))\|^2_{L^{\infty}(0,T;L^2(\O))}\leq C,
\quad\quad\quad\quad\quad
a \|\rho\|^\gamma_{L^{\infty}(0,T;L^{\gamma}(\O))}\leq C,
\\&\|\mathbb T_\mu\|^2_{L^2(0,T;L^2(\O))}\leq C,
\quad\quad\quad
(\kappa(1-\kappa)+r)\|\sqrt\rho \nabla s(\rho)\|^2_{L^{\infty}(0,T;L^2(\O))}\leq C,
\\&
\kappa\|\sqrt{\mu'(\rho)\rho^{\gamma-2}}\nabla\rho\|^2_{L^2(0,T;L^2(\O))}\leq C, 
\end{split}
\end{equation}
and
\begin{equation}\label{priorie estimate2}
\begin{split}
\\&\delta\|\rho\|^{10}_{L^{\infty}(0,T;L^{10}(\O))}\leq C,\quad\quad\;\;\;\quad\quad\quad\quad\delta\|\sqrt{\mu'(\rho)\rho^{8}}\nabla\rho\|^2_{L^2(0,T;L^2(\O))}\leq C,
\\&r_2\|(\frac{\rho}{\mu'(\rho_n)})^{\frac{1}{4}}\u\|^4_{L^4(0,T;L^4(\O))}\leq C,
\quad\quad\quad r_1\|\rho^{\frac{1}{3}}|\u|\|^3_{L^3(0,T;L^3(\O))}\leq C,
\\&r_0\|\u\|^2_{L^2(0,T;L^2(\O))}\leq C,
\quad\quad\quad\quad\quad\quad\quad
    r \|\mathbb S_r\|^2_{L^2(0,T;L^2(\O))} \leq C.
\end{split}
\end{equation}
Note that  the bounds \eqref{priori estimates} provide the following control on the velocity field
$$ \|\sqrt{\rho}\, \u\|^2_{L^{\infty}(0,T;L^2(\O))}\leq C.$$
Moreover let $$\displaystyle Z (\rho)= \int_0^\rho \frac{\sqrt{\mu(s)}\mu'(s)}{s}\, ds\;\;\text{and }\; \displaystyle  Z_1(\rho) =  \int_0^\rho \frac{\mu'(s)}{(\mu(s))^{1/4} s^{1/2}} \, ds,$$
we have the extra control
\begin{equation}
\label{J inequality for sequence}
r \left[\int_0^T\int_{\O}|\nabla^2Z(\rho)|^2\,dx\,dt
+\int_0^T\int_{\O} |\nabla Z_1(\rho)|^4\,dx\,dt\right]
\leq C,
\end{equation}
and
\begin{equation}
\label{priori mu} 
\begin{split}
&\|\mu(\rho)\|_{L^\infty(0,T;W^{1,1}(\Omega))} + 
      \|\mu(\rho)\u\|_{L^\infty(0,TL^{3/2}(\O))\cap L^2(0,T;W^{1,1} (\O))} \leq C,\\
&     \|\partial_t \mu(\rho)\|_{L^{\infty}(0,T;W^{-1,1}(\O))}\leq C, \\
&  \|Z(\rho)\|_{L^\infty(0,T;L^{1+}(\Omega))}  +
  \|Z_1(\rho)\|_{L^\infty(0,T;L^{1+}(\Omega))} \leq C,
\end{split}
\end{equation}
where  $C>0$ is a constant which depends only on the initial data.
\end{Theorem}
\smallskip


\noindent {\bf Sketch of proof for Theorem \ref{main result 1}.}
To show Theorem \ref{main result 1}, we need to build the smooth solution to an approximation associated to \eqref{last level approximation}. Here, we adapt the ideas developed in \cite{BDZ} to construct this approximation. More precisely, we consider an augmented version of the system
which will be more appropriate to construct approximate solutions. Let us explain the idea.
\vskip0.1cm
\noindent{\it First step: the  augmented system.}
Defining a new velocity field generalizing the one introduced in the BD entropy estimate namely
$$\w=\u+ 2\kappa\nabla s(\rho)$$
and a drift velocity
$\v =2  \nabla s(\rho)$
and ${s}(\rho)$ defined in \eqref{s}.

Assuming to have a smooth solution of \eqref{last level approximation} with damping terms, it ca{v}own that $(\rho,\w,{v})$
satisfies the following system of equations
$$\rho_t + {\rm div}(\rho \w) - 2\kappa \Delta \mu(\rho) = 0$$
and
\begin{equation*}
\begin{split}
\\&(\rho\w)_t+\Dv(\rho\u\otimes\w)-2(1-\kappa)\Dv(\mu(\rho)\mathbb{D}\, \w)
 -2\kappa\Dv(\mu(\rho)\mathbf{A}(w))
\\&- (1-\kappa) \nabla(\lambda(\rho) \Dv (w-\kappa {v}))
+\nabla\rho^{\gamma}+\delta\nabla\rho^{10}
   +4(1-\kappa)\kappa \Dv (\mu(\rho)\nabla^2{s}(\rho))
\\&=- r_0 (w-2\kappa \nabla s(\rho))
 -  r_1 \rho|\w-2\kappa\nabla{s}(\rho)|(\w-2\kappa\nabla{s}(\rho))\\&
 - r_2 \frac{\rho}{\mu'(\rho)}  |\w-2\kappa\nabla{s}(\rho)|^2(\w-2\kappa\nabla{s}(\rho))
+r\rho\nabla\left(\sqrt{K(\rho)}\D(\int_0^{\rho}\sqrt{K(s)}\,ds)\right),
\end{split}
\end{equation*}
and
\begin{equation*}
\begin{split}
& (\rho{v})_t+\Dv(\rho\u\otimes{v})-2\kappa\Dv(\mu(\rho)\nabla {v})
+ 2\Dv(\mu(\rho)\nabla^t\w) + \nabla(\lambda(\rho)\Dv(\w-\kappa {v}))=0,
\end{split}
\end{equation*}
where
$${v}= 2 \nabla {s}(\rho), \qquad \w=\u+\kappa {v}$$
and 
$$ K(\rho) = 4 (\mu'(\rho))^2 / \rho .$$
This is the augmented version for which we will show that there exists global weak
solutions, adding an hyperdiffusivity $\varepsilon_2[ \Delta^{2s}\w -\Dv((1+|\nabla w|^2)\nabla w)]$  on the equation satisfied by $w$, and passing to the limit  $\varepsilon_2$ goes to zero.

\medskip

\noindent {\bf Important remark.} Note that recently Bresch-Couderc-Noble-Vila \cite{BCNV} showed the following interesting relation
\begin{equation*}
\rho\nabla\left(\sqrt{K(\rho)}\D(\int_0^{\rho}\sqrt{K(s)}\,ds)\right)
  =\Dv(F(\rho)\nabla^2 \psi(\rho))+
  \nabla\left((F'(\rho)\rho- F(\rho))\D\psi(\rho)\right),
\end{equation*}
with $F'(\rho)=\sqrt{K(\rho)\rho}$ and $\sqrt\rho \psi'(\rho) = \sqrt{K(\rho)}.$
  Thus choosing
$$F(\rho)=2\,\mu(\rho) \hbox{ and therefore }  F'(\rho)\rho- F(\rho)=\lambda(\rho),$$
this gives $\psi(\rho) = 2 {s}(\rho)$ and thus
\begin{equation}
\label{BCNV relationship}
\rho\nabla\left(\sqrt{K(\rho)}\D(\int_0^{\rho}\sqrt{K(s)}\,ds)\right)=
2 \Dv\Bigl(\mu(\rho)\nabla^2\bigl(2 {s}(\rho)\bigr)\Bigr)
   +\nabla\Bigl(\lambda(\rho)\D\bigl(2{s}(\rho)\bigr)\Bigr). 
\end{equation}
This identity will play a crucial role in the proof. It defines the appropriate capillarity term
to consider in the approximate system. Other identities will be used to define the weak  solution for the 
Navier-Stokes-Korteweg system and to pass  to the limit in it
 namely
 \begin{equation}\label{rel}
\begin{split}
&  2\mu(\rho)\nabla^2(2{\mathbf s}(\rho))
    + \lambda(\rho) \Delta (2{\mathbf s}(\rho))  = 4 \Bigl[2 \sqrt{\mu(\rho)} \nabla\nabla Z(\rho) 
    - \nabla (\sqrt{\mu(\rho)} \nabla Z(\rho)\Bigr] \\
& \hskip3cm +  (\frac{2\lambda(\rho)}{\sqrt{\mu(\rho)}}+ k(\rho))\Delta Z(\rho)\,  {\rm Id}
      - {\rm div} [ k(\rho)\nabla Z(\rho)]\,  {\rm Id}. 
\end{split}
\end{equation}
where $\displaystyle Z(\rho) = \int_0^\rho [(\mu(s))^{1/2} \mu'(s)]/s \, ds$ and
$\displaystyle k(\rho) = \int_0^\rho \frac{\lambda(s)\mu'(s)}{\mu(s)^{3/2}} ds.$  

Note that  the case considered in \cite{LaVa,VY-1, VY} is related
 $\mu(\rho) = \rho$ and $K(\rho) = 4/\rho$ which corresponds to
 the quantum Navier-Stokes system.  Note that two very interesting papers
 have been written by Antonelli-Spirito in \cite{AnSp0, AnSp} considering Navier-Stokes-Korteweg
 systems without such relation between the shear viscosity and the
 capillary coefficient. 

\begin{Remark}
The additional pressure $\delta\rho^{10}$ is used in \eqref{control preasuer}  thanks to $3\alpha_2-2\leq 10$.
\end{Remark}
\bigskip

\noindent {\it Second Step and main result concerning the compressible Navier-Stokes system.}
  To prove global existence of weak solutions of the compressible Navier-Stokes equations, we follow
the strategy introduced in \cite{LaVa, VY}.  To do so, first we approximate the viscosity
$\mu$ by a viscosity $\mu_{\varepsilon_1}$ such that $\inf_{s\in [0,+\infty)} \mu_{\varepsilon_1}'(s)\ge \varepsilon_1 >0$.
Then we use Theorem \ref{main result 1} to construct a $\kappa$ entropy weak solution
to the approximate system \eqref{last level approximation}. We then show that this $\kappa$ entropy weak solution is a renormalized solution of \eqref{last level approximation} in the sense introduced
in \cite{LaVa}. More precisely we prove the following theorem:
\begin{Theorem} \label{renorm} Let $\mu(\rho)$ verifies \eqref{regmu}--\eqref{mu estimate1}, $\lambda(\rho)$ given by \eqref{BD relationship}. If $r_0>0$, then we assume also that ${\rm inf}_{s \in [0,+\infty)} \mu'(s)=\epsilon_1 >0$. 
   Assume that  $r_1$ is small enough compared to $r$ and $r_2$  is small enough compared to $\delta$, the initial values verify
and
 \begin{equation}\label{Initial conditions}
 \begin{split}
 &   \int_\Omega \left(\rho_0\left(\frac{|\u_0+ 2\kappa \nabla s(\rho_0)|^2}{2}+(\kappa (1-\kappa)+r)\frac{|2\nabla s(\rho_0)|^2}{2}\right) \right)\, dx\\
 &\hskip4cm   +\int_\Omega 
   \left(a \frac{\rho_0^\gamma}{\gamma-1}+ \mu(\rho_0) +\delta \frac{\rho^{10}}{9}+\frac{r_0}{\varepsilon_1}|(\ln \rho_0)_-|\right)\,dx <+\infty. 
 \end{split}
 \end{equation} 
Then the $\kappa$ entropy weak solutions is a renormalized solution of \eqref{last level approximation} in the sense of 
Definition \ref{def_renormalise_u}.
\end{Theorem} 
 We then pass to the limit with respect to the parameters $r,r_0,r_1,r_2$ and $\delta$
to recover a renormalized weak solution of the compressible Navier-Stokes equations and prove our main theorem.

\vskip0.3cm
   \textbf{Definitions}.
   Following \cite{LaVa} (based on the work in \cite{VY}), we will  show  the existence of renormalized solutions in $\u$. Then, we will show that this renormalized solution  is a  weak solution. The renormalization provides weak stability of the advection terms $\rho \u\otimes \u$ together and
 $\rho \u\otimes \v$. Let us first define the renormalized solution:

\vskip0.3cm
\begin{Definition}\label{def_renormalise_u} Consider $\mu>0$, $3\lambda +2 \mu>0$, $r_0\geq0$, $r_1\geq0$, $r_2\ge 0$ and $r\geq0$.
We say that $(\sro,\sro\u)$ is a  renormalized weak  solution in $\u$, if  it verifies  \eqref{priori estimates}-\eqref{priori mu}, and
 for  any function $\vfi\in W^{2,\infty}(\R^d)$ with  $\varphi(s)s \in L^{\infty}(\R^d)$, there exists three measures $R_{\vfi}, \overline{R}^1_\vfi,  \overline{R}^2_\vfi \in \mathcal{M}(\R^+\times\O)$, with
 $$
 \|R_{\vfi}\|_{ \mathcal{M}(\R^+\times\O)}+ \|\overline{R}^1_{\vfi}\|_{ \mathcal{M}(\R^+\times\O)}
 +  \|\overline{R}^2_{\vfi}\|_{ \mathcal{M}(\R^+\times\O)} \leq C \|\vfi''\|_{L^\infty(\R)},
 $$
 where the constant $C$ depends only on the solution $(\sro,\sro \u)$, and
 for any function $\psi\in \Cin$,
\begin{eqnarray*}
&&\int_0^T \int_\O \left(\rho \psi_t + \sqrt \rho  \sqrt \rho \u\cdot \nabla\psi \right)dx\, dt=0,\\
&&\int_0^T \int_\O \bigl( \rho \vfi(\u) \psi_t + \rho \vfi(\u)\otimes  \u :\nabla \psi \bigr)  \> dx\, dt\\
&& \hskip.2cm -  \int_0^T \int_\Omega  \left( 2 (\sqrt{\mu(\rho)}  \Sn
+  \frac{\lambda(\rho)}{2\mu(\rho)} {\rm Tr} (\sqrt{\mu(\rho)}\mathbb S_\mu) {\rm Id}) \,  \vfi'(\u)
     \right)\cdot   \nabla\psi \, dx dt \\
&& \hskip.2cm - \,  r \int_0^T \int_\Omega \left(2(\sqrt{\mu(\rho)}  \mathbb S_r
+  \frac{\lambda(\rho)}{2\mu(\rho)} {\rm Tr} (\sqrt{\mu(\rho)}\mathbb S_r) {\rm Id}\bigr)\,  \vfi'(\u)
           \right) \cdot \nabla\psi  \, dx dt\\
&&\hskip7cm +F(\rho,\u)\, \vfi'(\u) \psi \, dx\, dt=\left \langle R_{\vfi}, \psi\right\rangle, \\
&& \int_0^T \int_\O (\mu(\rho) \psi_t + \frac{\mu(\rho)}{\sqrt \rho} \sqrt \rho \u \cdot \nabla \psi) \, dx dt  - \int_0^T  
      \int_\Omega  \frac{\lambda(\rho)}{2\mu(\rho)} {\mathrm Tr} (\sqrt{\mu(\rho)} \Tn) 
             \psi  \, dx dt = 0,
\end{eqnarray*}
where $\Sn$ is given in \eqref{Smu} and $\mathbb{T}_\mu$ is given in \eqref{Tmu1}.
The matrix $\mathbb S_r$ is compatible in \eqref{eq_quantic}, \eqref{ZZ}, and \eqref{eq_quantic11}.

The vector valued function $F$ is given by 
\begin{equation}\label{eq_F}
\begin{split}
F(\rho,\u) 
& = \sqrt{\frac{P'(\rho) \rho }{\mu'(\rho)}} \nabla \int_0^\rho \sqrt{\frac{P'(s)\mu'(s)}{s}}\, ds \\
&  \hskip.5cm +  \delta\sqrt{\frac{P_\delta'(\rho) \rho }{\mu'(\rho)}} 
                \nabla \int_0^\rho \sqrt{\frac{P_\delta'(s)\mu'(s)}{s}}\, ds
     -r_0 \u - r_1 \rho|\u|\u-\frac{r_2}{\mu'(\rho)}\rho|\u|^2\u.
 \end{split}
\end{equation}
 For every $i,j,k$ between 1 and $d$:
\begin{equation}\label{eq_viscous_renormaliseAAA}
\sqrt{\mu(\rho)}\vfi_i'(\u)[\Tn]_{jk}= \partial_j(\mu(\rho)\rho\vfi'_i(\u)\u_k)
    -\sro\ u_k\vfi'_i(\u) \sqrt\rho \partial_j s(\rho)+ \overline{R}^1_\vfi,
\end{equation}
\begin{equation}\label{eq_kortweg_renormalise}
r\vfi_i'(\u)[\nabla(\sqrt{\mu(\rho)} \nabla Z(\rho))]_{jk}= 
r\partial_j(\sqrt{\mu(\rho)} \vfi'_i(\u)\partial_k Z(\rho))+ \overline{R}^2_\vfi,
\end{equation}
and
$$\|\overline{R}^1_\vfi\|_{\mathcal{M}(\R^+\times\O)} +
\|\overline{R}^2_\vfi\|_{\mathcal{M}(\R^+\times\O)} +
\|R_\vfi\|_{\mathcal{M}(\R^+\times\O)}
\leq C\|\vfi''\|_{L^\infty}.
$$
and  for any $\overline{\psi}\in C^\infty_c(\O)$:
\begin{eqnarray*}
&&\lim_{t\to0}\int_\O \rho(t,x)\overline{\psi}(x)\,dx=\int_\O \rho_0(x)\overline{\psi}(x)\,dx,\\
&&\lim_{t\to0}\int_\O \rho(t,x)\u(t,x)\overline{\psi}(x)\,dx=\int_\O  m_0 (x)\overline{\psi}(x)\,dx,\\
&& \lim_{t\to0}\int_\O \mu(\rho)(t,x)\overline{\psi}(x)\,dx=\int_\O \mu(\rho_0)(x)\overline{\psi}(x)\,dx
\end{eqnarray*}
\end{Definition}

\bigskip

\noindent  
  We define a global weak solution of the approximate system or the compressible Navier-Stokes equation (when $r=r_0=r_1=r_2=\delta=0$) as follows
\begin{Definition}\label{defweak} 
Let ${\mathbb S}_\mu$ the symmetric part of $\mathbb {T}_\mu$ in $L^2((0,T)\times \O)$  verifying \eqref{Smu}--\eqref{Tmu1}  and  $\mathbb{S}_r$  the capillary quantity in  $L^2((0,T)\times \O)$  given by \eqref{eq_quantic}--\eqref{eq_quantic11}. Let us denote
$P(\rho) = a \rho^\gamma$ and $P_\delta (\rho) = \delta \rho^{10}$.
We say that $(\rho,\u)$ is a weak solution to \eqref{last level approximation}--\eqref{ZZ}, if it satisfies the {\it a priori} estimates \eqref{priori estimates}--\eqref{priori mu} and for any function $\psi \in {\mathcal C}_c^\infty ((0,T)\times \Omega)$ 
verifying 
\begin{equation}
\begin{split}
& \int_0^T \int_\Omega (\rho \partial_t \psi + \rho \u \cdot \nabla \psi) \, dxdt= 0, \\
&\int_0^T \int_\Omega (\rho \u\partial_t \psi + \rho \u\otimes \u : \nabla \psi )\, dx dt \\
&    \hskip1.5cm  -  \int_0^T \int_\Omega  2 ( \sqrt{\mu(\rho)} \mathbb{S}_\mu  
              + \frac{\lambda(\rho)}{2\mu(\rho)} {\rm Tr} (\sqrt{\mu(\rho)} \mathbb S_\mu) 
                    {\rm Id}) \cdot  \nabla\psi \, dx dt  \\
&    \hskip1.5cm  -   r \int_0^T \int_\Omega  2 ( \sqrt{\mu(\rho)} \mathbb{S}_r  
              + \frac{\lambda(\rho)}{2\mu(\rho)} {\rm Tr} (\sqrt{\mu(\rho)} \mathbb S_r) 
                    {\rm Id}) \cdot\nabla\psi  \, dx dt \\
&  \hskip7cm + F(\rho,\u)  \, \psi \, dx dt = 0,\\
&  \int_0^\infty \int_\O \left(\mu(\rho) \psi_t + \frac{\mu(\rho)}{\sqrt \rho} \sqrt \rho \u \cdot \nabla \psi\right)
     dx \, dt  \\
&\hskip5cm   
  - \int_0^T \int_\Omega  \frac{\lambda(\rho)}{2\mu(\rho)}
      {\mathrm Tr} (\sqrt{\mu(\rho)}\mathbb T_\mu) \psi \, dx dt = 0,
\end{split}
\end{equation} 
with $F$ given through \eqref{eq_F}  and for any $\overline \psi \in {\mathcal C}_c^\infty(\O)$:
\begin{eqnarray*}
&&\lim_{t\to0}\int_\O \rho(t,x)\overline{\psi}(x)\,dx=\int_\O \rho_0(x)\overline{\psi}(x)\,dx,\\
&&\lim_{t\to0}\int_\O \rho(t,x)\u(t,x)\overline{\psi}(x)\,dx=\int_\O  m_0 (x)\overline{\psi}(x)\,dx,\\
&& \lim_{t\to0}\int_\O \mu(\rho)(t,x)\overline{\psi}(x)\,dx=\int_\O \mu(\rho_0)(x)\overline{\psi}(x)\,dx.
\end{eqnarray*}
\end{Definition}

\bigskip

\noindent {\bf Remark.}  As mentioned in \cite{BrGiLa}, the equation on $\mu(\rho)$ is important: By taking $\psi= {\rm div} \varphi$
for all $\varphi \in {\mathcal C}_0^\infty$, we can write the equation satisfied by $\nabla \mu(\rho)$
namely
\begin{equation} \label{grad}
\begin{split}
\partial_t \nabla\mu(\rho) +  {\rm div}(\nabla\mu(\rho) \otimes \u)   
& =
      {\rm div}(\nabla\mu(\rho) \otimes \u)    - \nabla {\rm div} (\mu(\rho) \u) \\
& \hskip3cm - \nabla\bigl( \frac{\lambda(\rho)}{2\mu(\rho)} {\rm Tr} (\sqrt{\mu(\rho)}{\mathbb T}_\mu)\Bigr) \\
& = - {\rm div}(\sqrt{\mu(\rho)} {}^t{\mathbb T}_\mu)   - \nabla\bigl( \frac{\lambda(\rho)}{2\mu(\rho)} {\rm Tr} (\sqrt{\mu(\rho)}{\mathbb T}_\mu)\Bigr). \\
\end{split}
\end{equation}
This will justify in some sense the two-velocities formulation introduced in \cite{BDZ} with the 
extra velocity linked to $\nabla\mu(\rho)$.


\section{The first level of approximation procedure}
The goal of this section is to construct a sequence of approximated solutions satisfying the compactness structure  to prove Theorem \ref{main result 1} namely the existence of weak solutions
of the  approximation system with capillarity and drag terms. Here we present the first level of approximation procedure.

\medskip

\noindent  1. The continuity equation
\begin{equation}
\label{approximation of the continuity equation}
\begin{split}&\rho_t+\Dv(\rho[\w]_{\varepsilon_3})=2\kappa\Dv\left([\mu'(\rho)]_{\varepsilon_4}\nabla\rho\right),
\end{split}
\end{equation}
with modified initial data
$$\rho(0,x)=\rho_0\in C^{2+\nu}(\bar{\O}), \quad0<\underline{\rho}\leq \rho_0(x)\leq \bar{\rho}.$$ 
Here $\varepsilon_3$ and $\varepsilon_4$ denote the standard regularizations by mollification with respect to space and time. 
This is a parabolic equation recalling that in this part ${\rm Inf}_{[0,+\infty)} \mu'(s) >0$. Thus, we can apply the standard theory of parabolic equation to solve it when $\w$ is given smooth enough. In fact,  the exact same equation was solved in paper \cite{BDZ}. In particular, we are able to get the following bound on the density at this level approximation
\begin{equation}
\label{low and upper bound on density}
0<\underline{\rho}\leq \rho(t,x)\leq\bar{\rho}<+\infty.
\end{equation}

\medskip

 \noindent 2. The momentum equation with drag terms is replaced by its Faedo-Galerkin approximation with the additional regularizing term $\varepsilon_2[ \Delta^{2s}\w -\Dv((1+|\nabla w|^2)\nabla w)]$ where $s\ge 2$
 \begin{equation}
 \begin{split}
 \label{approximation of the momentum equation}
&\int_{\O}\rho\w\cdot\psi\,dx-\int_0^t\int_{\O}\left(\rho([\w]_{\varepsilon_3}-2\kappa\frac{[\mu'(\rho)]_{\varepsilon_4}}{\rho}\nabla\rho)\otimes\w\right):\nabla\psi\,dx\,dt
\\&+2(1-\kappa)\int_0^t\int_{\O}\mu(\rho)\mathbb{D}\w:\nabla\psi\,dx\,dt+2\kappa\int_0^t\int_{\O}\mu(\rho)\mathbf{A}(w):\nabla\psi\,dx\,dt
\\&+(1-\kappa)\int_0^t\int_{\O}\lambda(\rho)\Dv\w\Dv\psi\,dx\,dt-2\kappa(1-\kappa)\int_0^t\int_{\O}\mu(\rho)\nabla\v:\nabla\psi\,dx\,dt
\\&-\kappa(1-\kappa)\int_0^t\int_{\O}\lambda(\rho)\Dv\v\Dv\psi\,dx\,dt-\int_0^t\int_{\O}\rho^{\gamma}\Dv\psi\,dx\,dt
-\delta\int_0^t\int_{\O}\rho^{10}\Dv\psi\,dx\,dt
\\&+\varepsilon_2\int_0^t\int_{\O}\left( \Delta^s\w\cdot\Delta^s\psi+(1+|\nabla \w|^2)\nabla\w:\nabla \psi\right)\,dx\,dt=
-\int_0^t\int_{\O} r_0 (\w-2\kappa\nabla{s}(\rho))\cdot\psi\,dx\,dt
 \\ & -r_1\int_0^t\int_{\O} \rho|\w-2\kappa\nabla{s}(\rho)|(\w-2\kappa\nabla{s}(\rho))\cdot\psi\,dx\,dt
\\&-r_2\int_0^t\int_{\O}\frac{\rho}{\mu'(\rho)}|\w-2\kappa\nabla{s}(\rho)|^2(\w-2\kappa\nabla{s}(\rho))\cdot\psi\,dx\,dt
\\&-r\int_0^t\int_{\O} \sqrt{K(\rho)}\D(\int_0^{\rho}\sqrt{K(s)}\,ds)\Dv (\rho\psi)\,dx\,dt+\int_{\O}\rho_0\w_0\cdot\psi\,dx
 \end{split}
 \end{equation}
satisfied for any $t>0$ and any test function $\psi\in C([0,T],X_n)$,
 where $\lambda (\rho)=  2(\mu'(\rho)\rho-\mu(\rho))$, 
 and ${s}'(\rho)= \mu'(\rho) /\rho
$,  and  $X_n=\text{span}\{e_i\}_{i=1}^{n}$ is an orthonormal basis in $W^{1,2}(\O)$ with $e_i\in C^{\infty}(\O)$ for any integers $i>0$.

 \medskip

3. The Faedo-Galerkin approximation for the equation on the drift velocity ${v}$
reads
\begin{equation}
\begin{split}
\label{artificial equation}
&\int_{\O}\rho\v\cdot\phi\,dx-\int_0^t\int_{\O}(\rho ([\w]_{\varepsilon_3}-2\kappa\frac{[\mu'(\rho)]_{\varepsilon_4}}{\rho} \nabla\rho)\otimes\v):\nabla\phi\,dx\,dt
\\&+2\kappa\int_0^t\int_{\O}\mu(\rho)\nabla\v:\nabla\phi\,dx\,dt
+ \kappa\int_0^t\int_{\O} \lambda(\rho)\Dv\v\, \Dv\phi\,dx\,dt
\\&-\int_0^t\int_{\O}
\lambda(\rho)\Dv\w\Dv\phi\,dx\,dt
+2\int_0^t\int_{\O}\mu(\rho)\nabla^T\w:\nabla\phi\,dx\,dt
=\int_{\O}\rho_0\v_0\cdot\phi\,dx
\end{split}
\end{equation}
satisfied for any $t>0$ and any test function $\phi\in C([0,T],Y_n)$, where $Y_n=\text{span}\{b_i\}_{i=1}^n$ and $\{b_i\}_{i=1}^{\infty}$ is an orthonormal
basis in $W^{1,2}(\O)$ with $b_i\in C^{\infty}(\O)$ for any integers $i>0.$\\

The above full approximation is similar to the ones in \cite{BDZ}.  We can repeat the same argument as their paper to obtain the local existence of solutions to the Galerkin approximation. In order to extend the local solution  to the global one, the uniform bounds are necessary so that the corresponding procedure can be iterated.

\subsection{The energy estimate if the solution is regular enough.}
 For any fixed $n>0,$
choosing  test functions $\psi=\w, \,\phi=\v$ in \eqref{approximation of the momentum equation} and \eqref{artificial equation}, we find that 
 $(\rho,\w,\v)$ satisfies the following $\kappa-$entropy equality
\begin{equation}
\label{entropy for first level approximation}
\begin{split}
&\int_{\O}\left(\rho\left(\frac{|\w|^2}{2}+(1-\kappa)\kappa\frac{|\v|^2}{2}\right)+\frac{\rho^{\gamma}}{\gamma-1}+\delta\frac{\rho^{10}}{9}\right)\,dx
+2(1-\kappa)\int_0^t
 \int_{\O}\mu(\rho)|\mathbb{D}\w-\kappa\nabla\v|^2\,dx\,dt
\\
&+ (1-\kappa)\int_0^t\int_{\O} \lambda(\rho)(\Dv\w-\kappa\Dv\v)^2\,dx\,dt+
+2\kappa\int_0^t\int_{\O}\frac{\mu'(\rho)p'(\rho)}{\rho}|\nabla\rho|^2\,dx\,dt
\\&+2\kappa\int_0^t\int_{\O}\mu(\rho)|A\w|^2\,dx\,dt+\varepsilon_2\int_0^t\int_{\O}\left( |\Delta^s\w|^2+(1+|\nabla \w|^2)|\nabla \w|^2\right)\,dx\,dt
\\&+ r \int_0^t \int_{\O} \sqrt{K(\rho)} \Delta (\int_0^\rho \sqrt{K(s)}\, ds) {\rm div}(\rho w) \,dx\,dt
+20\kappa\int_0^t\int_{\O}\mu'(\rho)\rho^8|\nabla\rho|^2\,dx\,dt
\\& 
+ r_0  \int_0^t \int_\O (w-2\kappa\nabla{s}(\rho))\cdot w \, dx\,dt 
+ r_1 \int_0^t\int_{\O}  \rho|\w-2\kappa\nabla{s}(\rho)|(\w-2\kappa\nabla{s}(\rho))\cdot\w\,dx\,dt
\\&+ r_2 \int_0^t\int_{\O}\frac{\rho}{\mu '(\rho)}|\w-2\kappa\nabla{s}(\rho)|^2(\w-2\kappa\nabla{s}(\rho))\cdot\w\,dx\,dt
\\&= \int_{\O}\left(\rho_0\left(\frac{|\w_0|^2}{2}+(1-\kappa)\kappa\frac{|\v_0|^2}{2}\right)+\frac{\rho_0^{\gamma}}{\gamma-1}+\delta\frac{\rho_0^{10}}{9}\right)\,dx-\int_0^T\int_{\O}\rho^{\gamma}\Dv([\w]_{\varepsilon_3}-\w)\,dx\,dt
\\&-\delta\int_0^T\int_{\O}\rho^{10}\Dv([\w]_{\varepsilon_3}-\w)\,dx\,dt,
\end{split}
\end{equation}
where ${s}'= {\mu'(\rho)}/{\rho}$ and $p(\rho)=\rho^{\gamma}.$
Compared to the calculations made in \cite{BDZ},  we have to take care of the capillary term and then to  take care of the drag terms showing that they can be controlled using that $\int_{s\in [0,T]} \mu'(s) \ge \varepsilon_1$ for the linear drag, using the extra pressure term $\delta \rho^{10}$ for the quadratic drag term and using the capillary term $r \rho \nabla(\sqrt{K(\rho)} \Delta (\int_0^\rho \sqrt{K(s)})$ for the cubic drag term. To do so, let us provide some properties on the capillary term and rewrite the terms coming from the drag quantities.

\medskip

\subsubsection{Some properties on the capillary term}

 Using the mass equation, the capillary term  in the entropy estimates reads
\begin{equation}
\begin{split}
&\int_\Omega  \sqrt{K(\rho)}
  \Delta(\int_0^\rho \sqrt{K(s)} \, ds)\,  \Dv(\rho w)
=  \frac{r}{2} \frac{d}{dt}\int_\Omega |\nabla \int_0^\rho \sqrt{K(s)} \, ds|^2 \\
   &     + 2\kappa \int_{\O}
    \sqrt{K(\rho)}  \Delta(\int_0^\rho \sqrt{K(s)} \, ds)\,  \Delta \mu(\rho) = I_1 + I_2 .
\end{split}
\end{equation}
In fact, we write term $I_1$ as follows $$\frac{r}{2} \frac{d}{dt}\int_\Omega |\nabla \int_0^\rho \sqrt{K(s)} \, ds|^2 =\frac{r}{2} \frac{d}{dt}\int_\Omega \rho|\nabla{s}(\rho)|^2\,dx. $$
By  \eqref{BCNV relationship}, we have
\begin{equation}
\begin{split}
I_2 &=  \int_{\O}
   \sqrt{K(\rho)}  \Delta(\int_0^\rho \sqrt{K(s)} \, ds)\,  \Delta \mu(\rho)
\\&=  - \int_{\O}  \rho \nabla \Bigl( \sqrt{K(\rho)}  \Delta(\int_0^\rho \sqrt{K(s)} \, ds)\Bigr)
   \cdot \nabla {s}(\rho)
   \\&
=   \int_{\O}  2 \mu(\rho) |2\nabla^2 {s}(\rho)|^2 + \lambda(\rho)|2\Delta {s}(\rho)|^2.
\end{split}
\end{equation}
\smallskip

\noindent {\it Control of norms using $I_2$.} Let us first recall that since
$$\lambda(\rho) = 2(\mu'(\rho)\rho- \mu(\rho)) > -2\mu(\rho)/3,$$
there exists $\eta >0$ such that
$$
2 \int_0^T\int_{\O}\mu(\rho)|\nabla^2{s}(\rho)|^2\,dx\,dt
+ \int_0^T\int_{\O}\lambda(\rho)|\Delta{s}(\rho)|^2\,dx\,dt$$
$$\hskip3cm \ge \eta \Bigl[ 
2 \int_0^T\int_{\O}\mu(\rho)|\nabla^2{s}(\rho)|^2\,dx\,dt
+ \frac{1}{3}\int_0^T\int_{\O}\mu(\rho)|\Delta{s}(\rho)|^2\,dx\,dt \Bigr].
$$
As the second term in the right-hand side is positive, lower bound on the quantity 
\begin{equation} \label{kor}
\int_0^T\int_{\O}\mu(\rho)|\nabla^2{s}(\rho)|^2\,dx\,dt
\end{equation}
will provide the same lower bound on $I_2$.

Let us now precise the norms which are controlled by \eqref{kor}. To do so, we need to rely on the following lemma on the density. In this lemma, we prove a more general entropy dissipation inequality than the one introduced by J\"ungel in \cite{J} and more general than those by J\"ungel-Matthes in \cite{JuMa}.
\begin{Lemma}
\label{Lemma on jungel type inequality} Let $\mu'(\rho)\rho<k\mu(\rho)$ for $2/3<k<4$ and 
$${s}(\rho)= \int_0^\rho \frac{\mu'(s)}{s} \, ds, \qquad
    Z(\rho) =\int_0^\rho \frac{\sqrt{\mu(s)}}{s}\mu'(s)\, ds, \qquad
    Z_1(\rho) = \int_0^\rho \frac{\mu'(s)}{(\mu(s))^{1/4}s^{1/2}}   \, ds.$$
{\rm i)} Assume $\rho>0$ and $\rho\in L^2(0,T;H^2(\Omega))$   then there exists $\varepsilon(k) >0$, such that we have the following estimate
\begin{equation*}
\int_0^T\int_{\O}|\nabla^2Z(\rho)|^2\,dx\,dt+\varepsilon(k)\int_0^T\int_{\O}\frac{\rho^2}{\mu(\rho)^{3}}|\nabla Z(\rho)|^4\,dx\,dt
\leq \frac{C}{\varepsilon(k)} \int_0^T\int_{\O}\mu(\rho)|\nabla^2{s}(\rho)|^2\,dx\,dt,
\end{equation*}
where $C$ is a  universal positive constant. 

\noindent {\rm ii)} Consider a sequence of smooth densities $\rho_n>0$ such that $Z(\rho_n)$ and $Z_1(\rho_n)$ converge strongly in $L^1((0,T)\times\Omega)$ respectively to $Z(\rho)$ and $Z_1(\rho)$ and $\sqrt{\mu(\rho_n)} \nabla^2 {\mathbf s}(\rho_n)$
is uniformly bounded in $L^2((0,T)\times\Omega)$. Then
\begin{equation*}
\int_0^T\int_{\O}|\nabla^2Z(\rho)|^2\,dx\,dt+\varepsilon(k)\int_0^T\int_{\O}|\nabla Z_1(\rho)|^4\,dx\,dt
\leq C < +\infty
\end{equation*}
\end{Lemma}

\begin{Remark} The case of $Z=2\sqrt{\rho}$ for
the inequality was proved in \cite{J}, which is critical to derive the uniform bound on approximated velocity in $L^2(0,T;L^2(\O))$ in \cite{VY-1,VY}.  The above lemma will play a similar role in this paper.
\end{Remark}
\begin{proof} Let us first prove the part i). Note that  $Z'(\rho)=\frac{\sqrt{\mu(\rho)}}{\rho}\mu'(\rho)$, we  get the following calculation:
\begin{equation*}
\begin{split}
\label{key-1}
\sqrt{\mu(\rho)}\nabla^2s(\rho)&=\sqrt{\mu(\rho)}\nabla(\frac{\nabla\mu(\rho)}{\rho})=\sqrt{\mu(\rho)}\nabla\left(\frac{1}{\sqrt{\mu(\rho)}}\nabla Z(\rho)\right)
\\&=\nabla^2 Z(\rho)-\frac{\nabla Z(\rho)}{\sqrt{\mu(\rho)}}\otimes\nabla\sqrt{\mu(\rho)}
\\&=\nabla^2 Z(\rho) -\frac{\rho\nabla Z(\rho)\otimes \nabla Z(\rho)}{2 \mu(\rho)^{\frac{3}{2}}}.
\end{split}
\end{equation*}
Thus, we have \begin{equation}
\begin{split}
\label{key-2}
\int_{\O}\mu(\rho)|\nabla^2s(\rho)|^2\,dx&=\int_{\O}|\nabla^2Z(\rho)|^2\,dx
   +\frac{1}{4}\int_{\O}\frac{\rho^2}{\mu(\rho)^3}|\nabla Z(\rho)|^4\,dx
\\&-
\int_{\O}\frac{\rho}{\mu(\rho)^{\frac{3}{2}}}\nabla^2Z(\rho) :(\nabla Z(\rho)\otimes \nabla Z(\rho))\,dx.
\end{split}
\end{equation}
By integration by parts, the cross product term reads as follows
\begin{equation}
\begin{split}
\label{key-3}&-\int_{\O}\frac{\rho}{\mu(\rho)^{\frac{3}{2}}}\nabla^2Z(\rho):(\nabla Z(\rho)\otimes \nabla Z(\rho))\,dx \\
& =
-\int_{\O}\frac{\rho\sqrt{\mu(\rho)}}{\mu(\rho)}\nabla^2Z(\rho):(\frac{\nabla Z(\rho)}{\sqrt{\mu(\rho)}}\otimes \frac{\nabla Z(\rho)}{\sqrt{\mu(\rho)}})\,dx
\\&=\int_{\O}\frac{\rho}{\mu(\rho)}\sqrt{\mu(\rho)}\nabla Z(\rho)\cdot\Dv(\frac{\nabla Z(\rho)}{\sqrt{\mu(\rho)}}\otimes \frac{\nabla Z(\rho)}{\sqrt{\mu(\rho)}})\,dx \\
& \hskip1cm +\int_{\O}\nabla(\frac{\rho}{\sqrt{\mu(\rho)}})\otimes\nabla Z(\rho):\frac{\nabla Z(\rho)\otimes \nabla Z(\rho)}{\mu(\rho)}\,dx
\\&=I_1+I_2.
\end{split}
\end{equation}
To this end, we are able to control $I_1$ directly,
\begin{equation}\label{key-I1}\begin{split}
|I_1|&\leq \varepsilon\int_{\O}\frac{\rho^2}{\mu(\rho)^3}|\nabla Z(\rho)|^4\,dx
   + \frac{C}{\varepsilon} \int_{\O}\mu(\rho)|\nabla(\frac{\nabla Z(\rho)}{\sqrt{\mu(\rho)}})|^2\,dx
\\&\leq \varepsilon\int_{\O}\frac{\rho^2}{\mu(\rho)^3}|\nabla Z(\rho)|^4\,dx
+ \frac{C}{\varepsilon}\int_{\O}\mu(\rho)|\nabla^2 s(\rho)|^2\,dx,
\end{split}
\end{equation}
where $C$ is a universal positive constant.
We calculate $I_2$ to have
\begin{equation}
\label{key-I2}
\begin{split}
I_2&=\int_{\O}\nabla(\frac{\rho}{\sqrt{\mu(\rho)}})\otimes\nabla Z(\rho):\frac{\nabla Z(\rho)\otimes \nabla Z(\rho)}{\mu(\rho)}\,dx
\\&=\int_{\O}\frac{\nabla\rho\otimes\nabla Z(\rho)}{\mu(\rho)^{\frac{3}{2}}}:\left(\nabla Z(\rho)\otimes \nabla Z(\rho)\right)\,dx \\
&\hskip2cm  -\int_{\O}\frac{\rho}{\mu(\rho)^2}\nabla\sqrt{\mu(\rho)}\otimes \nabla Z(\rho):\left(\nabla Z(\rho)\otimes \nabla Z(\rho)\right)\,dx
\\&=\int_{\O}\frac{\rho}{\mu(\rho)^2\mu(\rho)'}|\nabla Z(\rho)|^4\,dx-\frac{1}{2}\int_{\O}\frac{\rho^2}{\mu(\rho)^3}|\nabla Z(\rho)|^4\,dx.
\end{split}
\end{equation}
Relying on \eqref{key-2}-\eqref{key-I2}, we have
\begin{equation*}
\begin{split}&
\int_{\O}|\nabla^2Z(\rho)|^2\,dx+\int_{\O}\frac{\rho}{\mu(\rho)^2\mu'(\rho)}|\nabla Z(\rho)|^4\,dx
-(\frac{1}{4}+\varepsilon)\int_{\O}\frac{\rho^2}{\mu(\rho)^3}|\nabla Z(\rho)|^4\,dx
\\&\leq \frac{C}{\varepsilon} \int_{\O}\mu(\rho)|\nabla^2 s(\rho)|^2\,dx.
\end{split}
\end{equation*}
Since $k_1\mu'(s) s\leq \mu(s),$ we have
$$\frac{s}{\mu^2(s)\mu'(s)}-(\frac{1}{4}+\varepsilon)\frac{s^2}{\mu(s)^3}\geq (k_1-\frac{1}{4}-\varepsilon)\frac{s^2}{\mu(s)^3}>\varepsilon\frac{s^2}{\mu(s)^3},$$
 where we choose $k_1>\frac{1}{4}$.
 This implies
 $$\int_{\O}|\nabla^2Z(\rho)|^2\,dx+\varepsilon\int_{\O}\frac{\rho^2}{\mu(\rho)^3}|\nabla Z(\rho)|^4\,dx
\leq \frac{C}{\varepsilon} \int_{\O}\mu(\rho)|\nabla^2 s(\rho)|^2\,dx.$$
This ends the proof of part i). Concerning part ii), it suffices to pass to the limit in the inequality
proved previously using the lower semi continuity on the left-hand side.
\medskip

\end{proof}

\subsubsection{Drag terms control.} We have to discuss three kind of drag terms: Linear drag term, quadratic drag term and finally cubic drag term.

\medskip

\noindent {a)  \it Linear drag terms.} As in previous works \cite{BD,VY-1,Z}, we need to choose a linear drag with constant coefficient
\begin{equation}
\begin{split}
\label{11additional velocity term-control}
&r_0\int_0^t\int_{\O}(\w-2\kappa\nabla{s}(\rho))\cdot\w\,dx\,dt
=r_0\int_0^t\int_{\O}|\w-2\kappa\nabla{s}(\rho)|^2\,dx\,dt \\&
+r_0\int_0^t\int_{\O}(\w-2\kappa\nabla{s}(\rho))
    \cdot(2\kappa\nabla{s}(\rho))\,dx\,dt.
\end{split}
\end{equation}
The second term on the right side of \eqref{11additional velocity term-control} reads
\begin{equation*}
\begin{split}r_0\int_0^t\int_{\O}(\w-2\kappa\nabla{s}(\rho))&\cdot(2\kappa\nabla{s}(\rho))\,dx\,dt
=r_0\int_0^t\int_{\O}\rho(\w-2\kappa\nabla{s}(\rho))\cdot\frac{2\kappa\nabla{s}(\rho)}{\rho}\,dx\,dt
\\&
=r_0\int_0^t\int_{\O}\rho(\w-2\kappa\nabla{s}(\rho))\cdot2\kappa\nabla g(\rho)\,dx\,dt\\&
=r_0\int_0^t\int_{\O}\rho_t g(\rho)\,dx\,dt,
\end{split}
\end{equation*}
where $g'(\rho)= \frac{s'(\rho)}{\rho}=\frac{\mu'(\rho)}{\rho^2}$ and 
$g(\rho)=\int_1^\rho\frac{\mu'(r)}{r^2}\,dr.$
Letting $$G(\rho)=\int_1^\rho\int_1^r\frac{\mu'(\zeta)}{\zeta^2}\,d\zeta\,dr,$$ then 
$$r_0\int_{\O}\rho_t g(\rho)\,dx= r_0\frac{\partial}{\partial_t}\int_{\O}G(\rho)\,dx,$$
which implies
$$r_0\int_0^t\int_{\O}\rho_t g(\rho)\,dx\,dt= r_0\int_{\O}G(\rho)\,dx.$$
 Meanwhile, since $\lim_{\zeta\to 0}\mu'(\zeta)=\varepsilon_1>0$, for any $|\zeta|<\epsilon$ and any small number $\epsilon>0$, we have $\mu'(\zeta)\geq \frac{\varepsilon_1}{2}.$
 Thus, we have further estimate on $G(\rho)$ as follows
 \begin{equation*}
 \begin{split}
 G(\rho)=\int_1^\rho\int_1^r\frac{\mu'(\zeta)}{\zeta^2}\,d\zeta\,dr
 &\geq \frac{\varepsilon_1}{2}\int_1^\rho(1-\frac{1}{r})\,dr
 \\&= \frac{\varepsilon_1}{2}(\rho-1-\ln\rho)
 \\&\geq -\frac{\varepsilon_1}{4}(\ln\rho)_{-},
 \end{split}
 \end{equation*}
 for any $\rho\leq \epsilon$.
 Similarly, we can show that
 $$G(\rho)\leq 4\varepsilon_1(\ln\rho)_{+}$$ for any $\rho\leq \epsilon$.
 For given number $\epsilon_0>0$, if $\rho\geq \epsilon_0$, then we have $$0\leq G(\rho)\leq C\int_1^\rho\int_1^r\mu'(\zeta)\,d\zeta\,dr\leq C\mu(\rho)\rho.$$

\noindent {b) \it Quadratic drag term.} 
We use the same argument as in \cite{BDZ} to handle this term.  
The quadratic drag term gives
\begin{equation}
\begin{split}
\label{drag term control}
&r_1\int_0^t\int_{\O}
    \rho|\w-2\kappa\nabla{s}(\rho)|(\w-2\kappa\nabla{s}(\rho))\cdot\w\,dx\,dt
\\&=r_1\int_0^t\int_{\O} \rho
    |\w-2\kappa\nabla{s}(\rho)|^3\,dx\,dt
\\&\quad\quad\quad\quad+r_1\int_0^t\int_{\O} \rho|\w-2\kappa\nabla{s}(\rho)|(\w-2\kappa\nabla{s}(\rho))\cdot(2\kappa\nabla{s}(\rho))\,dx\,dt.
\end{split}
\end{equation} The second drag term of the right--hand side
can be controlled as follows
\begin{equation}
\label{the second term of drag term control}
\begin{split}
&r_1\left|\int_0^t\int_{\O}\rho|\w-2\kappa\nabla{s}(\rho)|(\w-2\kappa\nabla{s}(\rho))
\cdot(2\kappa\nabla{s}(\rho))\,dx\,dt\right|
\\&\leq r_1\int_0^t\int_{\O}\mu(\rho)|\u||\mathbb{D}\u|\,dx\,dt\\
&\leq \frac{1}{2}\int_0^t\int_{\O}\mu(\rho)|\mathbb{D}\u|^2\,dx\,dt
+\frac{r_1^2}{2}\int_0^t\int_{\O}\mu(\rho)|\u|^2\,dx\,dt,
\end{split}
\end{equation}
and $$\|\sqrt{\mu(\rho)}|\u|\|_{L^2(0,T;L^2(\O))}\leq C\|\rho^{\frac{1}{3}}|\u|\|_{L^3(0,T;L^3(\O))}\|\frac{\sqrt{\mu(\rho)}}{\rho^{\frac{1}{3}}}\|_{L^6(0,T;L^6(\O))}.$$
Note that
\begin{equation}
\begin{split}
&\int_0^t\int_{\O}\frac{\mu(\rho)^3}{\rho^2}\,dx\,dt=\int_0^t\int_{0\leq \rho\leq 1}\frac{\mu(\rho)^3}{\rho^2}\,dx\,dt+
\int_0^t\int_{\rho\geq 1}\frac{\mu(\rho)^3}{\rho^2}\,dx\,dt
\\&\leq C\int_0^t\int_{0\leq \rho\leq 1}\mu(\rho)(\mu'(\rho))^2\,dx\,dt+
\int_0^t\int_{\rho\geq 1}\frac{\mu(\rho)^3}{\rho^2}\,dx\,dt
\\&\leq C+\int_0^t\int_{\rho\geq 1}\frac{\mu(\rho)^3}{\rho^2}\,dx\,dt.
\end{split}
\end{equation}
From \eqref{mu estimate}, for any $\rho\geq 1$, we have
$$c'\rho^{\alpha_1} \leq \mu(\rho)\leq c\rho^{\alpha_2},$$
where $2/3<\alpha_1\leq \alpha_2<4.$
This yields to
\begin{equation}
\label{control preasuer}\int_0^t\int_{\rho\geq 1}\frac{\mu(\rho)^3}{\rho^2}\,dx\,dt\leq c\int_0^t\int_{\rho\geq 1}\rho^{3\,\alpha_2-2}\,dx\,dt
\leq c  \int_0^t \int_{\O}\rho^{10}\,dx
\end{equation}
for any time $t>0.$

\noindent {c) \it Cubic drag term.} The non-linear cubic drag term gives
\begin{equation}
\begin{split}
\label{drag term control}
&r_2\int_0^t\int_{\O}
    \frac{\rho}{\mu'(\rho) }|\w-2\kappa\nabla{s}(\rho)|^2(\w-2\kappa\nabla{s}(\rho))\cdot\w\,dx\,dt
\\&=r_2\int_0^t\int_{\O} \frac{\rho}{\mu'(\rho) }
    |\w-2\kappa\nabla{s}(\rho)|^4\,dx\,dt
\\&\quad\quad\quad\quad+r_2\int_0^t\int_{\O} \frac{\rho}{\mu'(\rho) }|\w-2\kappa\nabla{s}(\rho)|^2(\w-2\kappa\nabla{s}(\rho))\cdot(2\kappa\nabla{s}(\rho))\,dx\,dt.
\end{split}
\end{equation}
The novelty now is to show that we control the second drag term of the right--hand side
using the Korteweg-type information on the left-hand side
\begin{equation}
\label{the second term of drag term control}
\begin{split}
&r_2\int_0^t\int_{\O}\frac{\rho}{\mu'(\rho) }|\w-2\kappa\nabla{s}(\rho)|^2(\w-2\kappa\nabla{s}(\rho))\cdot(2\kappa\nabla{s}(\rho))\,dx\,dt
\\&\le r_2 \Bigl(
 \frac{3}{4} \int_0^t \int_{\O}  \frac{\rho}{\mu'(\rho)} |w-2\kappa \nabla {s}(\rho)|^4
 +  \frac{(2\kappa)^4}{4} \int_0^t \int_{\O}  \frac{\rho}{\mu'(\rho)} |\nabla {s}(\rho)|^4 \Bigr).
\end{split}
\end{equation}
Remark that the first term in the right-hand side may be absorbed using the first
term in \eqref{drag term control}. Let us now prove that if $r_1$ small enough, the
second term in the right-hand side may be absorbed by the term coming from
the capillary quantity in the energy. From Lemma
\ref{Lemma on jungel type inequality}, we have $$ \int_0^t\int_{\O}\frac{\rho^2}{\mu^{3}(\rho)}|\nabla Z(\rho)|^4\,dx\,dt=\int_0^t\int_{\O}\frac{1}{\mu(\rho)\rho^2}|\nabla\mu(\rho)|^4\,dx\,dt.
$$
It remains to check that
$$ \int_0^t \int_{\O}  \frac{\rho}{\mu'(\rho)} |\nabla {s}(\rho)|^4=
\int_0^t\int_{\O}\frac{1}{\mu'(\rho)\rho^3}|\nabla\mu(\rho)|^4\,dx\,dt\leq 
  C\int_0^t\int_{\O}\frac{1}{\mu(\rho)\rho^2}|\nabla\mu(\rho)|^4\,dx\,dt.$$
This concludes assuming $r_1$ small enough compared to $r$.

\subsubsection{The $\kappa$-entropy estimate.} Using the previous calculations,
assuming $r_2$ small enough compared to $r$, and denoting
$$ E[\rho,u+2\kappa \nabla \mathbf{s(\rho)}, \nabla  \mathbf{s(\rho)}]
= \int_\O\rho\left(\frac{|\u+2\kappa\nabla{s}(\rho)|^2}{2}
+ (1-\kappa)\kappa\frac{|\nabla{s}(\rho)|^2}{2}\right)+
\frac{\rho^{\gamma}}{\gamma-1}+\frac{\delta\rho^{10}}{9}+G(\rho),
$$
we get the following $\kappa$-entropy estimate
\begin{equation}
\label{entropy obtained}
\begin{split}
& E[\rho,u+2\kappa \nabla \mathbf{s(\rho)}, \nabla  \mathbf{s(\rho)}](t)
 +r_0\int_0^t\int_{\O}|\u|^2\,dx\,dt
\\&+\frac{r}{2}\int_{\O}|\nabla\int_0^{\rho}\sqrt{K(s)}\,ds|^2 \,dx
+2(1-\kappa)\int_0^t\int_{\O}\mu(\rho)|\mathbb{D}\u|^2\,dx\,dt+
20\kappa\int_0^t\int_{\O}\mu'(\rho)\rho^8|\nabla\rho|^2\,dx\,dt
\\
&+2(1-\kappa)\int_0^t\int_{\O}(\mu'(\rho)\rho-\mu(\rho))(\Dv\u)^2\,dx\,dt+2\kappa\int_0^t\int_{\O}\mu(\rho)|A(\u+2\kappa\nabla{s}(\rho))|^2\,dx\,dt
\\&+2\kappa\int_0^t\int_{\O}\frac{\mu'(\rho)p'(\rho)}{\rho}|\nabla\rho|^2\,dx\,dt
+ r_1 \int_0^t \int_{\O} \rho|\u|^3\, dx\,dt
     + \frac{r_2}{4} \int_0^t \int_{\O}  \frac{\rho}{\mu'(\rho)} |\u|^4 \, dx\, dt\\
&
+  \kappa r\int_0^t\int_{\O}\mu(\rho)|2 \nabla^2{s}(\rho)|^2\,dx\,dt
+ \frac{1}{2}\kappa  r\int_0^t\int_{\O}\lambda(\rho)|2\Delta{s}(\rho)|^2\,dx\,dt
\\&\leq \int_{\O}\left(\rho_0\left(\frac{|\w_0|^2}{2}+(1-\kappa)\kappa\frac{|\v_0|^2}{2}\right)+\frac{\rho_0^{\gamma}}{\gamma-1}+
\frac{\delta\rho_0^{10}}{9}+\frac{r}{2}|\nabla\int_0^{\rho_0} \sqrt{K(s)} \, ds|^2+G(\rho_0)\right)\,dx\\
& + C \frac{r_1}{\delta}  
     \int_\Omega E[\rho,u+2\kappa \nabla \mathbf{s(\rho)}, \nabla  \mathbf{s(\rho)}] dx \, dt .
\end{split}
\end{equation}
It suffices now to remark that
\begin{equation} \nonumber 
\begin{split}
& \int_0^t\int_\O \mu(\rho) | \mathbb{D}\u|^2 
  + \int_0^t \int_\O (\mu'(\rho)\rho - \rho) |{\rm div} \u|^2 \\
& = \int_0^t\int_\O \mu(\rho) | \mathbb{D}\u -\frac{1}{3} {\rm div} \u \, {\rm Id}|^2 \, dx dt 
  + \int_0^t \int_\O (\mu'(\rho)\rho -  \mu(\rho) + \frac{1}{3}\mu(\rho)) |{\rm div} \u|^2 .
\end{split}
\end{equation}
Note that $\alpha_1>2/3$,   there exists $\varepsilon>0$ such that 
$$\mu'(\rho)\rho - \frac{2}{3}\mu(\rho)  > \varepsilon \mu(\rho).$$
Such information and the control of $\sqrt{\mu(\rho)} |A(u)+2\kappa\nabla {\mathbf s}(\rho)|$
 in $L^2(0,T;L^2(\O))$ allow us, using the Gr\"onwall Lemma and the constraints on the parameters, to get the uniform estimates \eqref{priori estimates}--\eqref{J inequality for sequence}.
 
 \bigskip

Now we can show   \eqref{priori mu}.  First, we have 
$$\nabla \mu(\rho) = \frac{\nabla \mu(\rho)}{\sqrt \rho} \sqrt \rho
       \in  L^\infty(0,T;L^1(\Omega)),$$
due to the mass conservation and the uniform control on $\nabla\mu(\rho)/\sqrt\rho$
given in \eqref{priori estimates}.   Let us now write the equation satisfied by $\mu(\rho)$ namely
$$\partial_t\mu(\rho) + {\rm div}(\mu(\rho) \u) + \frac{ \lambda(\rho)}{2} {\rm div} \u = 0.$$
Recalling that $\lambda(\rho) = 2( \mu'(\rho)\rho - \mu(\rho))$ and the hypothesis on $\mu(\rho)$,
we get
$$\frac{d}{dt} \int_\Omega \mu(\rho) \le  
   C \, \bigl(\int_\Omega |\lambda(\rho)||{\rm div} \u|^2 + \int_\Omega \mu(\rho)\bigr),$$
and therefore 
$$\mu(\rho) \in L^\infty(0,T;L^1(\Omega)),$$
if $\mu(\rho_0) \in L^1(\Omega)$ due to the fact that $\sqrt{|\lambda(\rho)|}{\rm div} \u
\in L^2(0,T;L^2(\Omega)).$
Now, we observe that  $\mu(\rho)/\sqrt{\rho}$ is smaller than $1$ for $\rho\leq 1$ because
$\alpha_1 > 2/3$, and smaller than $\mu(\rho)$ for $\rho_n>1$, then $$\frac{\mu(\rho)}{\sqrt{\rho}} \in L^\infty(L^1).$$
Meanwhile, thanks to  \eqref{mu estimate}, we have 
$$
|\nabla( \mu(\rho)/\sqrt{\rho})|\leq \left|\frac{\nabla\mu(\rho)}{\sqrt{\rho}}\right|+\frac{\mu(\rho)}{2\rho\sqrt{\rho}}|\nabla\rho|\leq \left(1+\frac{1}{\alpha_1}\right)\left|\frac{\nabla\mu(\rho)}{\sqrt{\rho}}\right|.
$$
By \eqref{priori estimates},
 $\nabla(\mu(\rho)/\sqrt{\rho})$ is bounded in $L^\infty(0,T;L^2(\Omega))$ and finally 
$\mu(\rho)/\sqrt{\rho}$ is bounded in $L^\infty(0,T;(L^6(\Omega))$.
Thus, we have that 
$$\mu(\rho) \u = \frac{\mu(\rho)}{\sqrt\rho} \sqrt \rho\u,$$
is uniformly bounded in $ L^\infty(0,T;L^{3/2}(\O)).$
Let us come back to the equation satisfied by $\mu(\rho)$
which reads
$$\partial_t \mu(\rho) + {\rm div}(\mu(\rho) \u) + \frac{\lambda(\rho)}{2}{\rm div} \u= 0.$$
 Recalling that
$\lambda(\rho) {\rm div} \u \in L^\infty(0,T;L^1(\O))$, then we get
the conclusion on $\partial_t \mu(\rho)$. 
Let us now to prove that 
$$Z(\rho)= \displaystyle  \int_0^{\rho_n} \frac{\sqrt{\mu(s)} \mu'(s)}{s} ds \in L^{1+}((0,T)\times \O) 
     \hbox{ uniformly.} $$
Note first that
$$0 \le \frac{\sqrt{\mu(s)} \mu'(s)}{s} \le \alpha_2 \frac{\mu(s)^{3/2}}{s^2}
    \le c_2 \alpha_2(s^{3\alpha_1/2-2} 1_{s\le 1} +  \frac{\mu(s)^{3/2-}}{s^{2-}} 1_{s\ge 1}).
$$
There  exists $   \varepsilon>0 \hbox{ such that } \alpha_1 > 2/3+ \varepsilon,$
 thus
$$
0 \le \frac{\sqrt{\mu(s)} \mu'(s)}{s} \le 
c_2 \alpha_2 ( s^{\varepsilon -1}1_{s\le 1} + \frac{\mu(s)^{3/2-}}{s^{2-}} 1_{s\ge 1}).
$$
Note that $\mu'(s) > 0$ for $s>0$ and the definition of $Z(\rho)$, we get
$$0\le Z(\rho) \le C (\rho^\varepsilon  + \mu(\rho)^{3/2-})$$
with $C$ independent of $n$.
Thus $Z(\rho) \in L^{\infty}(0,T; L^{1+}(\O))$ uniformly with respect to $n$. 
Bound on $Z_1(\rho)$ follows the similar lines.

\subsection{Compactness Lemmas.} In this subsection, we provide  general compactness lemmas which will be used several times in this paper.

\bigskip

\noindent {\it Some uniform compactness.}  

\begin{Lemma}\label{compactuniforme} Assume we have a sequence $\{\rho_n\}_{n\in \mathbb N}$ satisfying  the estimates in Theorem \ref{main result 1}, uniformly with respect to $n$. Then, there exists a function $\rho \in L^\infty(0,T;L^\gamma(\O))$ such that,
 up to a subsequence,
$$ \mu(\rho_n) \to \mu(\rho) \hbox{ in } {\mathcal C}([0,T]; L^{3/2}(\O) \hbox{ weak}),$$
and
$$ \rho_n \to \rho \hbox{ a.e. in } (0,T)\times \O.$$
Moreover 
$$ \rho_n \to \rho \hbox{ in } L^{(4\gamma/3)^+}((0,T)\times \Omega),$$
\qquad 
$$\sqrt{\frac{P'(\rho_n)\rho_n}{\mu'(\rho_n)}} \nabla 
      \displaystyle \Bigl(\int_0^{\rho_n} \sqrt{\frac{P'(s)\mu'(s)}{s}}\, ds\Bigr)
\rightharpoonup \sqrt{\frac{P'(\rho)\rho}{\mu'(\rho)}} \nabla 
      \displaystyle \Bigl(\int_0^{\rho} \sqrt{\frac{P'(s)\mu'(s)}{s}}\, ds\Bigr)
   \hbox{ in } L^{1}((0,T)\times \O)
   $$
   and 
$$\sqrt{\frac{P'(\rho_n)\rho_n}{\mu'(\rho_n)}} \nabla 
      \displaystyle \Bigl(\int_0^{\rho_n} \sqrt{\frac{P'(s)\mu'(s)}{s}}\, ds\Bigr)
   \in L^{1+}((0,T)\times\Omega).$$
   If $\delta_n>0$ is such that $\delta_n\to \delta\geq 0$, then $$\delta_n\rho_n^{10}\to \delta\rho^{10}\quad\text{ in } L^{\frac{4}{3}}((0,T)\times\O).$$
\end{Lemma}
\bigskip

\noindent {\bf Proof.} From the estimate on $\mu(\rho_n)$ and Aubin-Lions lemma,  up to a subsequence, we have 
$$\mu(\rho_n)  \to \mu(\rho) \hbox{ in } {\mathcal C}([0,T]; L^{3/2}(\O) \hbox{ weak})$$
and therefore using that $\mu'(s)>0$ on $(0,+\infty)$ with $\mu(0)=0$, we get the conclusion 
on $\rho_n$.    Let us now recall that 
\begin{equation} \label{muineq}
\frac{\alpha_1}{\rho_n} \le  \frac{\mu'(\rho_n)}{\mu(\rho)} \le \frac{\alpha_2}{\rho_n} 
\end{equation} 
and therefore
$$ c_1 \rho_n^{\alpha_2} \le \mu(\rho_n) \le c_2 \rho_n^{\alpha_1}
     \qquad  \hbox{  for } \rho_n \le 1,$$
and 
$$c_1 \rho_n^{\alpha_1} \le \mu(\rho_n) \le c_2 \rho_n^{\alpha_2}
    \qquad \hbox{  for } \rho\ge 1.$$
with $c_1$ and $c_2$ independent on $n$.
Note that 
\begin{equation}\label{estimpressure}
\sqrt{\frac{p'(\rho_n)\mu'(\rho_n)}{\rho_n}}\nabla \rho_n \in L^\infty(0,T;L^2(\O))
   \hbox{ uniformly.} 
\end{equation}
Let us prove that there exists $\varepsilon$ such that
$$I_0= \displaystyle \int_0^T\int_\Omega \rho_n^{\frac{4\gamma}{3}+\varepsilon} <  C$$
with $C$ independent on $n$ and the parameters. We first remark that it suffices
to look at it when $\rho_n \ge 1$ and to remark there exists $\varepsilon$ such that
$\varepsilon \le (\gamma-1)/3.$ Let us take such parameter then
$$ \int_0^T\int_\Omega \rho_n^{\frac{4\gamma}{3}+\varepsilon} 1_{\rho \ge 1}
   \le \int_0^T\int_\Omega \rho_n^{\frac{2\gamma}{3} + \gamma - \frac{1}{3}} 1_{\rho \ge 1}
   \le \int_0^T \int_\O \rho_n^{\frac{2\gamma}{3} + \gamma + \alpha_1 -1} 1_{\rho \ge 1} 
$$
recalling that $\alpha_1 >2/3.$
Following  \cite{LiXi}, it remains to prove that
$$\displaystyle I_1=   \int_0^T\int_\O  \bigl[\rho_n^{[5\gamma + 3(\alpha_1-1)]/3}  \, 1_{\rho \ge 1} \bigr] <+\infty$$ uniformly. Denoting 
 $$I_2 = \int_0^T\int_\O  \bigl[\rho^{[5\gamma + 3(\alpha_2-1)]/3}  \, 1_{\rho \le 1} \bigr]$$
and using  the bounds on $\mu(\rho_n)$ in terms of power functions in $\rho$, which are different if $\rho_n \ge 1$ or $\rho_n\le 1$, we can write:
 $$ I_1 \le  I_1 + I_2 \le C_a \int_0^T \int_\O \rho_n^{2\gamma/3} P'(\rho_n) \,\mu(\rho_n)
      \le  C_a \int_0^T \|\rho_n^\gamma\|^{2/3}_{L^1(\O)}\|P'(\rho_n)\mu(\rho_n)\|_{L^3(\O)}$$
 where $C$ does not depend on $n$.
 Using the Poincar\'e-Wirtinger inequality, one obtains that
 \begin{equation*}
 \begin{split}\|P'(\rho_n)\mu(\rho_n)\|_{L^3(\O)} &= \|\sqrt{P'(\rho_n) \mu(\rho_n)}\|_{L^6(\O)}^2
     \\&\le  \|\sqrt{P'(\rho_n)\mu(\rho_n)}\|_{L^1(\O)} 
       + \|\nabla \bigl[\sqrt{P'(\rho_n)\mu(\rho_n)}\bigr]\|_{L^2(\O)}^2.
 \end{split}
 \end{equation*}
 Let us now check that the two terms are uniformly bounded in time.
 First we caculate
 $$\nabla \bigl[\sqrt{P'(\rho_n)\mu(\rho_n)}\bigr]
    = \frac{P''(\rho_n) \mu(\rho_n) 
       + P'(\rho_n)\mu'(\rho_n)}{\sqrt{P'(\rho_n)\mu(\rho_n)} }\nabla \rho_n$$
 and using \eqref{muineq}, we can  check that
 $$ \frac{P''(\rho_n) \mu(\rho_n) 
       + P'(\rho_n)\mu'(\rho_n)}{\sqrt{P'(\rho_n)\mu(\rho_n)} }
          \le  \sqrt{\frac{P'(\rho_n)\mu'(\rho_n)}{\rho_n}}.$$
 Therefore, using \eqref{estimpressure},  uniformly with respect to $n$, we get 
 $$ \sup_{t\in [0,T]} \|\nabla \bigl[\sqrt{P'(\rho_n)\mu(\rho_n)}\bigr]\|_{L^2(\O)}^2 < + \infty.$$
 Let us now check that uniformly with respect to $n$
 \begin{equation}\label{AAAestimm}
 \sup_{t\in [0,T]} \|\sqrt{P'(\rho_n)\mu(\rho_n)}\|_{L^1(\O)} < + \infty.
 \end{equation}
 Using the bounds on $\mu(\rho_n)$, we have
 $$\int_\O \sqrt{P'(\rho_n)\mu(\rho_n)}
      \le  C \int_\O \Bigl[\rho_n^{(\gamma-1+\alpha_1)/2} 1_{\rho_n \le 1}
           +  \rho_n^{(\gamma-1+\alpha_2)/2} 1_{\rho_n \ge 1} \Bigr]$$
with $C$ independent on $n$.
 Recalling that $\alpha_1 \ge 2/3$ and $\alpha_2 < 4$, we can check that            
  $$\int_\O \sqrt{P'(\rho_n)\mu(\rho_n)}
       \le  C \int_\O \Bigl[\rho_n^{\gamma/3}   +  \rho_n^{\frac{\gamma}{2}}\rho_n^{\frac{3}{2}}  \Bigr],$$
  and therefore using that $\rho_n^\gamma \in L^\infty(0,T;L^1(\O))$ and $\rho_n\in L^{\infty}(0,T;L^{10}(\O))$, we get
\eqref{AAAestimm}. This ends the proof of the convergence of $\rho_n$ to
$\rho$ in $L^{(4\gamma/3)^+}((0,T)\times \Omega$.

\medskip

\noindent Let us now focus on the convergence of 
\begin{equation}
\label{weak convergence of product}
\sqrt{\frac{P'(\rho_n)\rho_n}{\mu'(\rho_n)}} \nabla 
      \displaystyle \Bigl(\int_0^{\rho_n} \sqrt{\frac{P'(s)\mu'(s)}{s}}\, ds\Bigr).
      \end{equation}
First let us recall that 
$$\nabla  \displaystyle \Bigl(\int_0^{\rho_n} \sqrt{\frac{P'(s)\mu'(s)}{s}}\, ds\Bigr)
   \in L^\infty(0,T;L^2(\Omega)) \hbox{ uniformly}.$$
 Let us now prove that 
 \begin{equation}\label{estimm}
 \sqrt{\frac{P'(\rho_n)\rho_n}{\mu'(\rho_n)}}
    \in L^{2+}((0,T)\times \Omega).
\end{equation}
 Recall first that $\alpha_1 >\frac{2}{3}$, we just have to 
 consider $\rho_n \ge 1$. We write   
 $$\frac{P'(\rho_n)\rho_n}{\mu'(\rho_n)} 1_{\rho_n\ge 1}
     \le C \rho_n^{\gamma - \alpha_1 +1} 1_{\rho_n\ge 1}
     \le C \rho_n^{\gamma +1/3} 1_{\rho_n\ge 1}
     \le C \rho_n^{\frac{4\gamma}{3}} 1_{\rho_n \ge 1}.$$
  We can use the fact that  $\rho_n^{(4\gamma/3)^+} \in L^1((0,T)\times \Omega)$
  uniformly to conclude on \eqref{estimm}. Thanks to 
  $$\sqrt{\frac{P'(\rho_n)\rho_n}{\mu'(\rho_n)}}
     \to \sqrt{\frac{P'(\rho)\rho}{\mu'(\rho)}} \hbox{ in } L^2((0,T)\times \Omega)
  $$
  and
  $$ \nabla 
      \displaystyle \Bigl(\int_0^{\rho_n} \sqrt{\frac{P'(s)\mu'(s)}{s}}\, ds\Bigr)
       \to \nabla 
      \displaystyle \Bigl(\int_0^{\rho} \sqrt{\frac{P'(s)\mu'(s)}{s}}\, ds\Bigr)
      \hbox{ weakly in } L^2((0,T)\times \Omega),
  $$  
 we have the weak convergence of \eqref{weak convergence of product} in $ L^{1}((0,T)\times \O)$.

\bigskip
We now investigate limits on $\u$ independent of the parameters. We need to differentiate the case with hyper-viscosity $\eps_2>0$, from the case without. In the case with  hyper-viscosity, the estimate depends on $\eps_1$ because of the drag force $r_1$, while the estimate in the case $\eps_2=0$ is independent of all the other parameters. This is why we will consider the limit $\eps_2$ converges to 0 first.

\begin{Lemma}\label{lem u} Assume that $\eps_1>0$ is fixed. Then, there exists a constant $C>0$ depending on $\eps_1$ and $C_{in}$, but independent of all the other parameters (as long as they are bounded), such that for any initial values $(\rho_0, \sqrt{\rho_0}u_0)$ verifying (\ref{Initial conditions}) for  $C_{in}>0$
we have 
\begin{eqnarray*}
&&\|\partial_t(\rho \u)\|_{L^{1+}(0,T;W^{-s,2}(\Omega))}\leq C,\\
&&\|\nabla(\rho \u)\|_{L^2(0,T;L^1(\Omega))}\leq C.
\end{eqnarray*}

Assume  now that $\eps_2=0$.
Let $\Phi:\R^+\to \R$ be a smooth function, positive for $\rho>0$, such that 
\begin{eqnarray*}
&&\Phi(\rho)+|\Phi'(\rho)|\leq C e^{-\frac{1}{\rho}}, \qquad \mathrm{for} \ \rho\leq 1,\\
&&\Phi(\rho)+|\Phi'(\rho)|\leq C e^{-\rho}, \qquad \mathrm{for} \ \rho\geq 2.
\end{eqnarray*}
Assume that the initial values $(\rho_0, \sqrt{\rho_0}u_0)$ verify (\ref{Initial conditions}) for a fixed $C_{in}>0$.
Then, there exists a constant $C>0$ independent of $\eps_1, r_0, r_1, r_2, \delta$ (as long as they are bounded), such that 
\begin{eqnarray*}
&&\|\partial_t\left[\Phi(\rho) \u\right]\|_{L^{1+}(0,T;W^{-2,1}(\Omega))}\leq C,\\
&&\|\nabla\left[\Phi(\rho) \u\right]\|_{L^2(0,T;L^1(\Omega))}\leq C.
\end{eqnarray*}
\end{Lemma}
\begin{proof}
We split the proof into the two cases.
\vskip0.3cm \noindent {\bf Case 1:}  Assume that $\eps_1>0$. From the equation on $\rho u$ and the 
{\it a priori } estimates, we find directly that 
$$
\|\partial_t (\rho \u)\|_{L^{1+}(0,T;W^{-s,2}(\Omega))}\leq C+ r_1^{1/4} \frac{\|\rho\|^{1/4}_{L^1((0,T)\times\Omega)}}{\|\mu'(\rho)\|_{L^\infty((0,T)\times\Omega)}}\left(r_1\int_0^T\int_\Omega \rho |\u|^4\,dx\,dt \right)^{3/4}\leq C(1+1/\eps_1).
$$
We have $\mu(\rho)\geq \eps_1 \rho$, and from (\ref{priori estimates}), we have the {\it a priori} estimate 
$$
\|\nabla \sqrt{\rho}\|^2_{L^\infty(0,T;L^2(\Omega))}\leq \frac{C}{\eps_1}.
$$
Hence 
\begin{eqnarray*}
\|\nabla(\rho \u)\|_{L^2(0,T;L^1(\Omega))} 
&& \leq
\left\|\frac{\rho}{\sqrt{\mu}(\rho)}\right\|_{L^\infty(0,T;L^2(\Omega))} 
 \left\|\sqrt{\mu}(\rho)\nabla u\right\|_{L^2(0,T;L^2(\Omega)))} \\
 && \> +2\|\nabla \sqrt{\rho}\|_{L^\infty(0,T;L^2(\Omega))} \|\sqrt{\rho} \u\|_{L^\infty(0,T;L^2(\Omega))}\\
&&\> \leq C.
\end{eqnarray*}

\vskip0.3cm \noindent {\bf Case 2}: Assume now that $\eps_2=0$. Multiplying the equation on $(\rho u)$ by $\Phi(\rho)/\rho$, we get, as for the renormalization, that 
$$
\|\partial_t\left[\Phi(\rho)\u\right]\|_{L^{1+}(0,T;W^{-2,1}(\Omega))}\leq C.
$$
Note that 
\begin{eqnarray*}
&& \|\nabla\left[\Phi(\rho) \u\right]\|_{L^2(0,T;L^1(\Omega))}\leq
     \left\|\frac{\Phi(\rho)}{\sqrt{\mu}(\rho)}\right\|_{L^\infty}  \left\|\sqrt{\mu}(\rho)\nabla \u\right\|_{L^2(L^2)}\\
&&\qquad\qquad +2\| \frac{\Phi'(\rho)}{\mu'(\rho)}\|_{L^\infty((0,T)\times\Omega)} \|\mu'(\rho)\nabla \sqrt{\rho}\|_{L^\infty(0,T;L^2(\Omega))} \|\sqrt{\rho} \u\|_{L^\infty(0,T;L^2(\Omega))}\\
&&\qquad\qquad \leq C.
\end{eqnarray*}
\end{proof}

\begin{Lemma}\label{Compactnesstool1}
Assume either that $\eps_{2,n}=0$, or $\eps_{1,n}=\eps_1>0$. Let $(\rho_n,\sqrt{\rho_n} \u_n)$ be a sequence of solutions for a family of bounded parameters with uniformly bounded initial values verifying (\ref{Initial conditions}) with a fixed $C_{in}$.
Assume that there exists $\alpha>0$, and a smooth function $h:\R^+\times\R^3\to\R$ such that $\rho_n^\alpha$ is uniformly bounded in $L^p((0,T)\times\Omega)$ and  $h(\rho_n,\u_n)$ is uniformly bounded in $L^q((0,T)\times \Omega)$, with 
$$
\frac{1}{p}+\frac{1}{q}<1.
$$
Then, up to a subsequence, $\rho_n$ converges to a function $\rho$ strongly in $L^1$, $\sqrt{\rho_n}\u_n$ converges weakly to a function $q$ in $L^2$. We define $\u=q/\sqrt{\rho}$ whenever $\rho\neq 0$, and $\u=0$ on the vacuum where $\rho=0$. 
Then $\rho_n^\alpha h(\rho_n,\u_n)$ converges strongly in $L^1$ to $\rho^\alpha  h(\rho, \u)$. 
\end{Lemma}
\begin{proof}
Thanks to the uniform bound on the kinetic energy $\int \rho_n |\u_n|^2$, and to Lemma \ref{compactuniforme}, up to a subsequence, $\rho_n$ converges strongly in $L^1((0,T)\times \Omega)$ to a function $\rho$, and $\sqrt{\rho_n} \u_n$ converges weakly in $L^2((0,T)\times \Omega)$ to a function $q$. 

\vskip0.3cm
We want to show that, up to a subsequence, $\u_n {\bf 1}_{\{\rho>0\}}$ converges almost every where to $u {\bf 1}_{\{\rho>0\}}$. We consider the two cases. First, if $\eps_{1,n}=\eps_1>0$, then from Lemma \ref{lem u} and the Aubin-Lions Lemma,
$\rho_n \u_n$ converges strongly in $C^0(0,T; L^1(\Omega))$ to $\sqrt{\rho} q=\rho \u$.  Up to a subsequence, both $\rho_n$ and $\rho_n \u_n$ converges almost everywhere to, respectively,  $\rho$ and $\rho \u$. For almost every $(t,x) \in \{\rho>0\}$, for $n$ big enough, $\rho_n(t,x)>0$, so $\u_n=\rho_n \u_n/\rho_n$ at this point converges $u$. If $\eps_{2,n}=0$ we use the second part of Lemma \ref{lem u} and thanks to the Aubin-Lions Lemma, $\Phi(\rho_n)\u_n$ converges strongly  in $C^0(0,T; L^1(\Omega))$ to $\Phi(\rho) \u$. We still have, up to a subsequence, both $\rho_n$ and $\Phi(\rho_n) \u_n$ converging almost everywhere to, respectively,  $\rho$ and $\phi(\rho) \u$ (we used the fact that $\Phi(r)/\sqrt{r}=0$ at $r=0$). Since $\Phi(r)\neq 0$ for $r\neq0$, for almost every $(t,x) \in \{\rho>0\}$, for $n$ big enough, $\Phi(\rho_n)(t,x)>0$, so $u_n=\Phi(\rho_n) \u_n/\Phi(\rho_n)$ at this point converges $\u$. 

\vskip0.3cm
Note that 
$$
\rho_n^\alpha h(\rho_n,\u_n) =\rho_n^\alpha h(\rho_n,\u_n) {\bf 1}_{\{\rho>0\}}+\rho_n^\alpha h(\rho_n,\u_n) {\bf 1}_{\{\rho=0\}}.
$$
The first term converges almost everywhere to $\rho^\alpha h(\rho,\u) {\bf 1}_{\{\rho>0\}}$, and therefore to $\rho^\alpha h(\rho,\u) $ in $L^1$ by the Lebesgue's theorem. The second part can be estimated as follows
$$
\|\rho_n^\alpha h(\rho_n,\u_n) {\bf 1}_{\{\rho=0\}}\|_{{L^1}}\leq \|h(\rho_n,\u_n)\|_{L^q}\|\rho_n^\alpha  {\bf 1}_{\{\rho=0\}}\|_{{L^{p-\eps}}}.
$$
But $\rho_n^\alpha {\bf 1}_{\{\rho=0\}}$ converges almost everywhere to 0, by the Lebesgue's theorem, the last term converges to 0.
\end{proof}

\bigskip

\noindent {\it Some compactness when the  parameters are fixed.}  For any positive fixed $\delta$, $r_0$, $r_1$, $r_2$ and $r$, to recover a weak solution to \eqref{last level approximation}, we only need to handle the compactness of the terms 
$$r\rho_n\nabla\left(\sqrt{K(\rho_n)}\D(\int_0^{\rho_n}\sqrt{K(s)}\,ds)\right)$$ 
and
$$\frac{\rho_n}{\mu'(\rho_n)}|\u_n|^2\u_n.$$
Indeed due to the term $r_0\rho_n|\u_n|\u_n$ and the fact that $\inf_{s\in [0,+\infty)}\mu'(s) >\varepsilon_1>0$, one obtains the compactness for all other terms in the same way as   in \cite{BDZ,MV}.

\medskip

\noindent {\it Capillarity term.} To pass to the limits in
$$r\rho_n\nabla\left(\sqrt{K(\rho_n)}\D(\int_0^{\rho_n}\sqrt{K(s)}\,ds)\right),$$
we use the identity
\begin{equation}
\begin{split}
&\rho\nabla\left(\sqrt{K(\rho_n)}\D(\int_0^{\rho_n}\sqrt{K(s)}\,ds)\right) \\
& \hskip3cm =
   4 \Bigl[2{\rm div}(\sqrt{\mu(\rho_n)} \nabla\nabla Z(\rho_n)) 
     - \Delta (\sqrt{\mu(\rho_n)} \nabla Z(\rho_n)\Bigr]\\
&\hskip4cm + \Bigl[ \nabla \bigl[(\frac{2\lambda(\rho_n)}{\sqrt{\mu(\rho_n)}}
  + k(\rho_n))\Delta Z(\rho_n)\bigr]
      - \nabla {\rm div} [ k(\rho_n)\nabla Z(\rho_n)] \Bigr]
\end{split}
\end{equation}
where $\displaystyle Z(\rho_n) = \int_0^{\rho_n} [(\mu(s))^{1/2} \mu'(s)]/s \, ds$ and
$\displaystyle k(\rho_n)  = \int_0^{\rho_n} \frac{\lambda(s)\mu'(s)}{\mu(s)^{3/2}} ds.$  
It allows us to rewrite the weak form coming for the capillarity term as follows
 \begin{equation*}
\begin{split}
&\int_0^t\int_\O \sqrt{K(\rho_n)} \Delta (\int_0^{\rho_n} \sqrt{K(s)}\, ds) {\rm div} (\rho_n \psi) \, dx\, dt 
\\&=  4 \int_0^t\int_\O \bigl(2\sqrt{\mu(\rho_n)} \nabla\nabla Z(\rho_n): \nabla \psi 
    + \sqrt{\mu(\rho_n)}\nabla Z(\rho_n)\cdot \Delta \psi\bigr) \\
 & \hskip1cm 
    + \int_0^t\int_\Omega  \bigl(\frac{2\lambda(\rho_n)}{\sqrt{\mu(\rho_n)}} 
      + k(\rho_n))\Delta Z(\rho_n) \, {\rm div} \psi
        + k(\rho_n) \nabla Z(\rho_n) . \nabla {\rm div} \psi\bigr) 
\\
&=A_1+A_2.
\end{split}
\end{equation*}
In fact, with Lemma \ref{compactuniforme} at hand, we are able to have compactness of 
$A_1$ and $A_2$ easily. Concerning $A_1$, we know that
$$\sqrt{\mu(\rho_n)} \to \sqrt{\mu(\rho)} \hbox{ in } L^p((0,T); L^q(\O)) 
\hbox{ for all } p<+\infty \hbox{ and } q<3.$$
Note that $\nabla\nabla Z(\rho_n)$ is uniformly bounded in $L^2(0,T;L^2(\O))$, we have
$\nabla Z(\rho_n)$ is uniformly bounded in $L^2(0,T;L^6(\O))$, because 
$\int_\O \nabla Z(\rho_n) = 0$ due to the periodic condition. Thus we have following weak 
convergence
$$\int_{\O}\sqrt{\mu(\rho_n)}\nabla Z(\rho_n)\cdot \Delta\psi\,dx
        \to \int_{\O}\sqrt{\mu}\nabla Z\cdot \Delta\psi\,dx,$$ 
and 
$$\int_{\O}\sqrt{\mu(\rho_n)}\nabla \nabla Z(\rho_n)\nabla\psi\,dx
        \to \int_{\O}\sqrt{\mu}\nabla \nabla Z:\nabla\psi\,dx,$$ 
thanks to Lemma \ref{compactuniforme}. We conclude that $Z=Z(\rho)$, thanks to the bound
on $Z(\rho_n)$ and the strong convergence on $\rho_n$.
 Thus using the compactness on $\rho_n$,
 the passage to the limit in $A_1$ is done. Concerning $A_2$, we just have to look at the coefficients 
$$\displaystyle k(\rho_n)= \int_0^{\rho_n} \lambda(s)\mu'(s)/\mu(s)^{3/2} \, ds, \qquad
    j(\rho_n)= {2\lambda(\rho_n)}/{\sqrt{\mu(\rho_n)}}.$$
 Recalling the assumptions on $\mu(s)$ and the relation 
$\lambda(s) = 2 (\mu'(s)s -\mu(s))$, we have
$$2(\alpha_1- 1) \mu(s) \le \lambda(s) \le 2(\alpha_2-1) \mu(s),$$
and 
$$\frac{\alpha_1}{\sqrt{\mu(s)}s} \le \frac{\mu'(s)}{\mu(s)^{3/2}}
    \le \frac{\alpha_2}{\sqrt{\mu(s)}s}.$$ 
This means that the coefficients $k(\rho_n)$ and $j(\rho_n)$ are comparable to $\sqrt{\mu(\rho_n)}$. Using the compactness of the density $\rho_n$ and the informations on $\mu(\rho_n)$ given in Corollary \ref{compactuniforme}, we conclude the compactness of $A_2$  doing as for $A_1$. 
\medskip

\noindent {\it Cubic non-linear drag term.} 
We will use Lemma \ref{Compactnesstool1} to show the compactness of  $$\frac{\rho_n}{\mu'(\rho_n)}|\u_n|^2\u_n.$$
  More precisely, we write
\begin{equation}\label{decompdrag}
\frac{\rho_n}{\mu'(\rho_n)}|\u_n|^2\u_n=\rho_n^{\frac{1}{6}}\sqrt{\frac{\rho_n}{\mu'(\rho_n)}}|\u_n|^2\rho_n^{\frac{1}{3}}|\u_n|\frac{1}{\sqrt{\mu'(\rho_n)}}
= \rho_n^{1/6} h(\rho_n,|\u_n|),
\end{equation}
By Lemma \ref{compactuniforme}, there exists $\varepsilon>0$ such that $\rho_n^{\frac{1}{6}}$ is uniformly bounded in $L^{\infty}(0,T;L^{6\gamma+\varepsilon}(\O))$
and $\rho_n\to \rho\text{ a.e.}$, so 
\begin{equation}\label{comp1}
\rho_n^{\frac{1}{6}}\to\rho^{\frac{1}{6}}\quad\text{ in } L^{6\gamma+\varepsilon}((0,T)\times \O)).
\end{equation}
Note that  $\sqrt{\frac{\rho_n}{\mu'(\rho_n)}}|\u_n|^2$ is uniformly bounded in $L^2(0,T;L^2(\O))$,
and  $\inf_{s\in [0,+\infty)} \mu'(s) \ge \varepsilon_1 >0$, $\rho_n^{\frac{1}{3}}|\u_n|\frac{1}{\sqrt{\mu'(\rho_n)}}$ is uniformly bounded in $L^3(0,T;L^3(\O))$,
 thus
\begin{equation} \label{comp2}
h(\rho_n,|\u_n|) = 
\sqrt{\frac{\rho_n}{\mu'(\rho_n)}}|\u_n|^2\rho_n^{\frac{1}{3}}|\u_n|\frac{1}{\sqrt{\mu'(\rho_n)}}
\in L^{\frac{6}{5}}(0,T;L^{\frac{6}{5}}(\O))
   \hbox{ uniformly.}
\end{equation}
By Lemma \ref{Compactnesstool1}  and \eqref{decompdrag}--\eqref{comp2}, we deduce that 
$$\int_0^t\int_{\O}\frac{\rho_n}{\mu'(\rho_n)}|\u_n|^2\u_n\,dx\,dt\to \int_0^t\int_{\O}\frac{\rho}{\mu'(\rho)}|\u|^2\u\,dx\,dt. \, \square$$

Relying on the compactness stated in this section and the compactness in \cite{MV}, we are able to follow the argument in \cite{BDZ} to  show Theorem \ref{main result 1}. 
Thanks to  term $r_0\rho_n|\u_n|\u_n$, we have $$\int_0^T\int_{\O}r_0\rho_n|\u_n|^4\,dx\,dt\leq C.$$
This gives us that $$\sqrt{\rho_n}\u_n\to\sqrt{\rho}\u\; \text{ strongly in } L^2(0,T;L^2(\O)).$$
With above compactness of this section, we are able to pass to the limits for recovering a weak solution.
 In fact, to recover  a weak solution to \eqref{last level approximation}, we have to pass to the limits as the order of $\varepsilon_4\to 0$, $n\to\infty,$ $\varepsilon_3\to0$ and $\varepsilon\to 0$ respectively.  In particular, when passing to the limit $\varepsilon_3$ tends to zero, we also need  to handle the identification of $\v$ with 
$2\nabla{s}(\rho)$. 
 Following the same argument in \cite{BDZ}, one shows that $\v$ and $2\nabla{s}(\rho)$ satisfy the same moment equation. 
 By  the regularity and compactness of solutions, we can show the uniqueness of solutions.  By the uniqueness, we have $\v=2\nabla{s}(\rho)$. This ends the proof of Theorem \ref{main result 1}.

\section{From weak solutions to renormalized solutions to the approximation}

This section is dedicated to show that a weak solution is a renormalized solution for our last level of approximation namely to show Theorem \ref{renorm}. First, we introduce a new function
$$[f(t,x)]_\varepsilon =f*\eta_{\varepsilon}(t,x),\text{ for any\ \ } t>\varepsilon,\quad\text{ and }\;[f(t,x)]_\varepsilon^x =f*\eta_{\varepsilon}(x)$$ 
where $$\eta_{\varepsilon}(t,x)=\frac{1}{\varepsilon^{d+1}}\eta(\frac{t}{\varepsilon},\frac{x}{\varepsilon}),\quad\text{ and } \eta_{\varepsilon}(x)=\frac{1}{\varepsilon^{d}}\eta(\frac{x}{\varepsilon}),$$
with $\eta$ a smooth nonnegative even function compactly supported in the space time ball of radius 1, and with integral equal to 1.
   In this section, we will rely on the following two lemmas to proceed our ideas. 
Let $\partial$ be a partial derivative in one direction (space or time) in these two lemmas.  The first one is the commutator lemma of DiPerna and Lions, see \cite{Lions}.
 \begin{Lemma}
 \label{Lions's lemma} 
 Let $f\in W^{1,p}(\R^N\times\R^{+}),\,g\in L^{q}(\R^N\times\R^{+})$ with $1\leq p,q\leq \infty$, and $\frac{1}{p}+\frac{1}{q}\leq 1$. Then, we have
 $$\|
 [\partial(fg)]_\varepsilon -\partial(f([g]_{\varepsilon}))\|_{L^{r}(\R^N\times \R^+)}\leq C\|f\|_{W^{1,p}(\R^N\times\R^{+})}\|g\|_{L^{q}(\R^N\times\R^{+})}$$
 for some $C\geq 0$ independent of $\varepsilon$, $f$ and $g$, $r$ is determined by $\frac{1}{r}=\frac{1}{p}+\frac{1}{q}.$ In addition,
  $$[\partial(fg)]_{\varepsilon}-\partial(f([g]_{\varepsilon}))\to0\;\;\text{ in }\,L^{r}(\R^N\times\R^{+})$$
 as $\varepsilon \to 0$ if $r<\infty.$
 Moreover, in the same way if 
 $f\in W^{1,p}(\R^N),\,g\in L^{q}(\R^N)$ with $1\leq p,q\leq \infty$, and $\frac{1}{p}+\frac{1}{q}\leq 1$. Then, we have
 $$\|
 [\partial(fg)]^x_\varepsilon -\partial(f([g]^x_{\varepsilon}))\|_{L^{r}(\R^N)}\leq C\|f\|_{W^{1,p}(\R^N)}\|g\|_{L^{q}(\R^N)}$$
 for some $C\geq 0$ independent of $\varepsilon$, $f$ and $g$, $r$ is determined by $\frac{1}{r}=\frac{1}{p}+\frac{1}{q}.$ In addition,
  $$[\partial(fg)]^x_{\varepsilon}-\partial(f([g]^x_{\varepsilon}))\to0\;\;\text{ in }\,L^{r}(\R^N)$$
 as $\varepsilon \to 0$ if $r<\infty.$

 \end{Lemma}
We also need another very standard lemma as follows.
\begin{Lemma}
\label{standard lemma}
If $f\in L^p(\O\times\R^{+})$ and $g\in L^q(\O\times\R^{+})$ with $\frac{1}{p}+\frac{1}{q}=1$ and $H\in W^{1,\infty}(\R)$, then
\begin{equation*}
\begin{split}
&\int_0^T\int_{\O} [f]_\varepsilon g\,dx\,dt=\int_0^T\int_{\O}f [g]_\varepsilon \,dx\,dt,
\\&\lim_{\varepsilon\to 0 } \int_0^T\int_{\O} [f]_\varepsilon g\,dx\,dt=\int_0^T\int_{\O}f g\,dx\,dt,
\\&\partial [f]_\varepsilon =[\partial f]_\varepsilon,
\\&\lim_{\varepsilon\to 0}\|H([f]_\varepsilon)-H(f)\|_{L^s_{loc}}(\O\times\R^+)=0,\quad\text{for any }\, 1\leq s<\infty.
\end{split}
\end{equation*}
\end{Lemma}
   We define a nonnegative  cut-off functions $\phi_m$ for any fixed positive $m$ as follows.
\begin{equation}
\label{cutoff function}
\phi_m(y)\begin{cases}= 0, \;\;\;\;\;\quad\quad\quad\quad\text{ if }0\leq y\leq \frac{1}{2m},
\\ =2my-1,\;\;\;\;\;\quad\text{ if } \frac{1}{2m}\leq y\leq \frac{1}{m},
\\ =1,\,\;\;\;\;\quad\quad\;\quad\quad\text{ if }  \frac{1}{m}\leq y\leq m,
\\=2-\frac{y}{m},\,\;\;\;\;\quad\quad\text{ if }  m\leq y\leq 2m, 
\\=0,\,\;\;\;\;\quad\quad\;\quad\quad\text{ if }  y\geq 2m.
\end{cases}\end{equation}
It enables to define an approximated velocity for the density bounded away from zero and bounded away from infinity. It is crucial to process our procedure, since the gradient  approximated velocity is bounded in $L^2((0,T)\times \O)$.  In particular, we introduce $\u_m=\u\phi_m(\rho)$ for any fixed $m>0$. Thus, we can show 
$\nabla\u_m$ is bounded in $L^2(0,T;L^2(\O))$ due to \eqref{cutoff function}. In fact,
\begin{equation*}
\begin{split}
\nabla\u_m&=\phi_m'(\rho)\u\otimes\nabla\rho+\phi_m(\rho)\frac{1}{\sqrt{\mu(\rho)}}\Tn
\\&=\big(\phi_m'(\rho)\frac{(\mu(\rho)\rho)^{1/4}}{(\mu'(\rho))^{\frac{3}{4}}}\big)
\big((\frac{\rho}{\mu'(\rho)})^{\frac{1}{4}}\u\big)\otimes 
  \big(\frac{\mu'(\rho)}{\rho^{\frac{1}{2}}\mu(\rho)^{\frac{1}{4}}}\nabla\rho\big)
+\phi_m(\rho)\frac{1}{\sqrt{\mu(\rho)}}\Tn.
\end{split}
\end{equation*}
Similarly to \cite{LaVa}, thanks to the cut-off function \eqref{cutoff function} and for $m$ fixed,  $\phi_m'(\rho){(\mu(\rho)\rho)^{\frac{1}{4}}}/{(\mu'(\rho))^{\frac{3}{4}}}$ and  $\phi_m(\rho)/\sqrt{\mu(\rho)}$ are bounded. Then $\nabla\u_m$ is bounded in $L^2((0,T)\times \Omega)$
using the estimates with $r>0$ and $r_2>0$, and hence for $\varphi \in W^{2,+\infty}(\R)$, we get  $\nabla\varphi'((\u_m)_j)$ is bounded in $L^2((0,T)\times \Omega)$ for $j=1,2,3$.

\smallskip

The following estimates are necessary. We state them in the lemma as follows.
\begin{Lemma}
\label{estimate of approximation}
There exists  a constant $C>0$ depending only on the fixed solution $(\sqrt{\rho},\sqrt{\rho}\u)$, and $C_m$ depending also on $m$ such that
\begin{equation*}
\begin{split}&\|\rho\|_{L^{\infty}(0,T;L^{10}(\Omega))}
+\|\rho\u\|_{L^3(0,T;L^{\frac{5}{2}}(\O))}
+ \|\rho|\u|^2\|_{L^{2}(0,T; L^{\frac{10}{7}}(\O))}
 \\& +\|\sqrt{\mu}\big(|\Sn|+r|\Sk|\big)\|_{L^{2}(0,T; L^{\frac{10}{7}}(\O))}
 + \|\frac{\lambda(\rho)}{\mu(\rho)}\|_{L^{\infty}((0,T)\times \O)}
\\& + \|\sqrt{\frac{P'(\rho_n)\rho_n}{\mu'(\rho_n)}} \nabla 
      \displaystyle \Bigl(\int_0^{\rho_n} \sqrt{\frac{P'(s)\mu'(s)}{s}}\, ds\Bigr)\|_{L^{1+}((0,T)\times\Omega)} \\
&    + \|\sqrt{\frac{P_\delta'(\rho_n)\rho_n}{\mu'(\rho_n)}} \nabla 
      \displaystyle \Bigl(\int_0^{\rho_n} \sqrt{\frac{P_\delta'(s)\mu'(s)}{s}}\, ds\Bigr)\|_{L^{1+}((0,T) 
      \times\Omega)}
     +\|r_0\u\|_{L^2((0,T)\times \Omega)}\leq C,
\end{split}
\end{equation*}
and $$\|\nabla\phi_m(\rho)\|_{L^4((0,T)\times \Omega}+\|\partial_t\phi_m(\rho)\|_{L^2((0,T\times\Omega))}\leq C_m.$$
\end{Lemma} 
\begin{proof}

By \eqref{priorie estimate2}, we have   $\rho \in L^{\infty}(0,T;L^{10}(\O))$.
Now we have $\nabla\sqrt{\rho}\in L^{\infty}(0,T;L^2(\O))$ because $\mu'(s) \ge \varepsilon_1$ and 
$\mu'(\rho) \nabla \rho /\sqrt \rho \in L^\infty((0,T);L^2(\O))$. Note that 
$$\rho\u=\rho^{\frac{2}{3}}\rho^{\frac{1}{3}}\u,$$
$\rho^{\frac{2}{3}}\in L^{\infty}(0,T;L^{15}(\O))$ and $\rho^{\frac{1}{3}}\u\in L^{3}(0,T;L^3(\O))$, $\rho\u$ is bounded in $L^{3}(0,T;L^{\frac{5}{2}}(\O))$.  
\smallskip

By \eqref{priorie estimate2}, we have 
$(\frac{\rho}{\mu'(\rho)})^{1/2}|\u|^2\in L^2((0,T)\times \O)$.
Note that
$$\rho|\u|^2= (\rho\mu'(\rho))^{1/2} (\frac{\rho}{\mu'(\rho)})^{1/2}|\u|^2, $$
it is  bounded in $L^{2}(0,T;L^{\frac{10}{7}}(\O))$, where we used facts 
 that $\mu(\rho) \in L^\infty(0,T;L^{5/2}(\Omega))$ (recalling that for $\rho \ge 1$ we have $\mu(\rho)\le c\rho^4$ and 
$\rho \in L^\infty(0,T;L^{10}(\Omega))$) and  $\mu'(\rho) \rho \le \alpha_2 \mu(\rho)$.

Similarly, we get  $\sqrt{\mu}(|\Sn|+r|\Sk|) \in L^2(0,T;L^{10/7}(\Omega))$
by \eqref{priori estimates}.
 The $L^\infty((0,T)\times \O)$ bound for $\lambda(\rho)/\mu(\rho)$ may be obtained easily due to \eqref{BD relationship} and \eqref{mu estimate}.
 

\smallskip

Concerning the estimates related to the pressures, we just have to look at the proof in Lemma \ref{compactuniforme}.
Note that
 \begin{equation*}
\begin{split}
&\nabla\phi_m(\rho)=\phi_m'(\rho) \nabla\rho= \phi_m'(\rho)\frac{\rho^{1/2} \mu(\rho)^{1/4}}{\mu'(\rho)}   [\frac{\mu'(\rho)}{\rho^{1/2}\mu(\rho)^{1/4}}\nabla\rho]
\end{split}
\end{equation*}
by \eqref{J inequality for sequence}, we conclude that  $\nabla\phi_m(\rho)$ is bounded in $L^4((0,T)\times\O)$.
It suffices to recall that thanks to the cut-off function $\phi_m$, we have  $\phi_m'(\rho) \rho^{1/2} \mu(\rho)^{1/4}/\mu'(\rho)$ 
bounded in $L^{\infty}((0,T)\times \O)$. Similarly, we write 
\begin{equation*}
\begin{split}
\partial _t\phi_m(\rho)&=\phi_m'(\rho)\partial_t \rho=-\phi'_m(\rho)\Dv(\rho\u)
\\&=-\phi_m'(\rho)\frac{\rho}{\sqrt{\mu}}\mathrm{Tr} (\Tn)
 - \big(\phi_m'(\rho)\frac{(\mu(\rho)\rho)^{\frac{1}{4}}}{(\mu'(\rho))^{\frac{3}{4}}}\big)
\big(\frac{\rho^{\frac{1}{4}}}{(\mu'(\rho))^{\frac{1}{4}}}\u\big)\cdot
\big(\frac{\mu'(\rho)}{\rho^{1/2} \mu(\rho)^{1/4}} \nabla \rho \big)
,\end{split}
\end{equation*}
which provides  $\partial_t\phi_m(\rho)$ bounded in $L^2(0,T;L^2(\O))$ thanks to \eqref{priori estimates}, \eqref{priorie estimate2} and \eqref{J inequality for sequence}.
and using the cut-off function property to bound the extra quantiies in $L^\infty((0,T)\times\O)$ as previously.

\end{proof}

\begin{Lemma}
\label{Lemma of renormalized approxiamtion}
The $\kappa$-entropic weak solution constructed in Theorem \ref{main result 1} is a renormalized solution, in particular, we have 
\begin{equation}
\label{limit for m large}
\begin{split}
&
\int_0^T\int_{\O}\big(\rho\varphi(\u)\psi_t+ (\rho \varphi(\u)\otimes \u) \nabla\psi\big)\\
& -   \int_0^T\int_\O  \nabla\psi  \varphi'(\u)\big[2\bigl(\sqrt{\mu(\rho)}(\Sn+ r \, \Sk)
 + \frac{\lambda(\rho)}{2\mu(\rho)}{\rm Tr}(\sqrt{\mu(\rho)} \Sn + r \sqrt{\mu(\rho)} \Sk) {\rm Id} \big]\\
& -\int_0^T \int_\O \psi\varphi''(\u)\Tn \big[2\bigl((\Sn+ r \, \Sk)
 + \frac{\lambda(\rho)}{2\mu(\rho)}{\rm Tr}(\Sn + r  \Sk) {\rm Id} \big] \\
 & + \int_0^T \int_\O \psi\varphi'(\u)F(\rho,\u)\big)\,dx\,dt=0,
\end{split}
\end{equation}
where 
\begin{equation}\label{eq_viscous_renormalise}
\begin{split}
& \sqrt{\mu(\rho)}\varphi_i'(\u)[\Tn]_{jk}= \partial_j(\mu\varphi'_i(\u)\u_k)-\sro \u_k\varphi'_i(\u)\frac{\nabla\mu}{\sqrt{\rho}}+ \bar{R}^1_\varphi, \\
&\sqrt{\mu(\rho)} \varphi_i'(\u) [\mathbb S_r]_{jk}
      =  2 \sqrt{\mu(\rho)} \varphi_i'(\u) \partial_j \partial_k Z(\rho) 
        - 2 \partial_j (\sqrt{\mu(\rho)} \partial_k Z(\rho) \varphi_i'(\u))
         + \bar{R}^2_\varphi \\
&\frac{\lambda(\rho)}{2\mu(\rho)} \varphi_i'(\u) {\rm Tr} (\sqrt{\mu(\rho)} \mathbb T_\mu)
    = {\rm div} \bigl(\frac{\lambda(\rho)}{\mu(\rho)} \sqrt\rho \u \frac{\mu(\rho)}{\sqrt\rho}   \varphi'(\u) \bigr)   \\
 & \hskip4.4cm     - \sqrt \rho u \cdot \sqrt\rho \nabla s(\rho)\frac{\rho \mu''(\rho)}{\mu(\rho)}
        \varphi'(\u)+ \bar{R}^3_\varphi    \\
&  \frac{\lambda(\rho)}{\mu(\rho)} \varphi'(\u) {\rm Tr} (\sqrt{\mu(\rho)} \mathbb S_r) 
       = \varphi_i'(\u) \bigl(\frac{\lambda(\rho)}{\sqrt{\mu(\rho)}} + \frac{1}{2} k(\rho) \bigr) \Delta Z(\rho) \\
& \hskip4.4cm  - \frac{1}{2} {\rm div}(k(\rho) \varphi'_i(\u) \nabla Z(\rho))
        +  \bar{R}^4_\varphi 
\end{split}
\end{equation}
where 
\begin{equation}
\begin{split}
&\bar{R}^1_{\varphi}=\varphi''_i(\u)\Tn\sqrt{\mu(\rho)}\u \\
&\bar{R}^2_\varphi
         = 2 \varphi_i''(u) \mathbb T_\mu \nabla Z(\rho)\\
&\bar{R}^3_\varphi =  
- \varphi_i''(u) \mathbb T_\mu\cdot \sqrt{\mu(\rho)} \u \frac{\lambda(\rho)}{\mu(\rho)} \\
& \bar{R}^4_\varphi =
 \frac{k(\rho)}{2 \sqrt{\mu(\rho)}}  \varphi_i''(\u) \mathbb T_\mu \cdot \nabla Z(\rho)
\end{split}
\end{equation}
\end{Lemma}

\begin{proof}
We choose a function $\Bigl[\phi_m'([\rho]_\varepsilon)\psi\Bigr]_\varepsilon$ as a test function for the continuity equation with  $\psi\in C_c^{\infty}((0,T)\times\O)$. Using Lemma 
\ref{standard lemma}, we have 
 \begin{equation}
\begin{split}
\label{weak formulation for mass with varepsilon}
0&=\int_0^T\int_{\O}\big(\partial_t\Bigl[\phi_m'([\rho]_\varepsilon)\psi\Bigr]_\varepsilon \rho
   +\rho\u\cdot\nabla\Bigl[\phi_m'([\rho]_\varepsilon)\psi\Bigr]_\varepsilon\big)\,dx\,dt
\\&=-\int_0^T\int_{\O}\big(\phi_m'([\rho]_\varepsilon)\psi  \, \partial_t [\rho]_\varepsilon
    +\Dv([\rho\u]_{\varepsilon}) \phi_m'([\rho]_\varepsilon)\psi\big)\,dx\,dt
\\&=\int_0^T\int_{\O}\left(\psi_t\phi_m([\rho]_\varepsilon)
    -\psi\phi'_m([\rho]_\varepsilon)
     \bigl[\frac{\rho}{\sqrt{\mu(\rho)}}\mathrm{Tr} (\Tn)+2 \sqrt{\rho}\u\cdot\nabla\sqrt{\rho}\bigr]_\varepsilon\right)\,dx\,dt.
\end{split}
\end{equation}
Using Lemma \ref{estimate of approximation} and Lemma \ref{standard lemma}, and passing into the limit as $\varepsilon$ goes to zero,  from \eqref{weak formulation for mass with varepsilon}, we get: \begin{equation}
\begin{split}
\label{modified  continuity equation}
0&=\int_0^T\int_{\O}\big(\psi_t\phi_m(\rho)-\psi\phi'_m(\rho)[\frac{\rho}{\sqrt{\mu}}\mathrm{Tr} (\Tn)+2\sqrt{\rho}\u\cdot\nabla\sqrt{\rho}]\big)\,dx\,dt
\\&=\int_0^T\int_{\O}\big(\psi_t\phi_m(\rho)
-\psi \bigl[\phi'_m(\rho)\frac{\rho}{\sqrt{\mu}}\mathrm{Tr} (\Tn)+\u\cdot\nabla\phi_m(\rho)\bigr]\big)\,dx\,dt,
\end{split}
\end{equation}
thanks to $\psi\nabla\phi_m(\rho)\in L^4((0,T)\times \O)$, $\u\in L^2((0,T)\times \O)$, and $\psi $ compactly supported.

Similarly, we can choose $[\psi\phi_m(\rho)]_\varepsilon$ as a test function for the momentum equation. In particular, 
we have the following lemma.

\begin{Lemma}
\label{Lemma for limits-first two terms}
$$
\int_0^T\int_{\O} [\psi\phi_m(\rho)]_\varepsilon \big(\partial_t (\rho \u)  +\Dv(\rho\u\otimes\u)\big)\,dx\,dt$$
tends to 
$$-\int_0^T\int_{\O}\psi_t\rho\u_m+\nabla\psi\cdot(\rho\u\otimes\u_m
+\psi(\partial_t\phi_m(\rho)+\u\cdot\nabla\phi_m(\rho))\rho\u\,dx\,dt$$
as $\varepsilon\to 0.$
\end{Lemma}
\begin{proof}
By Lemma \ref{Lions's lemma},
we can show that  
\begin{equation*}
\begin{split}
&
\int_0^T\int_{\O} [\psi\phi_m(\rho)]_\varepsilon  \partial_t(\rho \u)\,dx\,dt\to 
-\int_0^T\int_{\O}\partial_t\psi \rho\u_m
+\psi\partial_t\phi_m(\rho) \rho \u\,dx\,dt.
\end{split}
\end{equation*}
For the second term, we have 
\begin{equation*}
\begin{split}
&\int_0^T\int_{\O} \bigl[\psi\phi_m(\rho)\bigr]_\varepsilon \Dv(\rho\u\otimes\u)\,dx\,dt
=\int_0^T\int_{\O}\psi\phi_m(\rho)\bigl[ \Dv(\rho\u\otimes\u)\bigr]_\varepsilon\,dx\,dt\\
&=\big(\int_0^T\int_{\O}\psi\phi_m(\rho) \bigl[ \Dv(\rho\u\otimes\u)\bigr]_\varepsilon\,dx\,dt
-\int_0^T\int_{\O}\psi\phi_m(\rho)\bigl[\Dv(\rho\u\otimes\u)\bigr]^x_\varepsilon\,dx\,dt\big)
\\&+\int_0^T\int_{\O}\psi\phi_m(\rho)\bigl[ \Dv(\rho\u\otimes\u)\bigr]^x_\varepsilon\,dx\,dt
\\&=R_1+R_2,
\end{split}
\end{equation*}
where $[f (t,x)]_\varepsilon =f(t,x)*\eta_{\varepsilon}(t,x)$ and $[f(t,x)]_\varepsilon^x =f*\eta_{\varepsilon}(x)$ with $\varepsilon>0$ a small enough number. 
We write $R_1$ in the following way
\begin{equation*}
\begin{split}
R_1&=\int_0^T\int_{\O}\psi\phi_m(\rho)\bigl[\Dv(\rho\u\otimes\u)\bigr]_\varepsilon\,dx\,dt
      -\int_0^T\int_{\O}\psi\phi_m(\rho)\bigl[\Dv(\rho\u\otimes\u)\bigr]_\varepsilon^x\,dx\,dt
\\&=\int_0^T\int_{\O}\psi\nabla\phi_m(\rho):\bigl[\rho\u\otimes\u\bigr]_\varepsilon\,dx\,dt
     -\int_0^T\int_{\O}\psi\nabla\phi_m(\rho):\bigl[\rho\u\otimes\u\big]^x_\varepsilon\,dx\,dt.
\end{split}
\end{equation*}
Thanks to Lemma \ref{estimate of approximation}, $\rho|\u|^2 \in L^{2}(0,T; L^{10/7}(\Omega))$ and $\psi\nabla\phi_m(\rho)\in L^4((0,T)\times \Omega)$, we conclude that $R_1\to 0$ as $\varepsilon\to0.$ Meanwhile, we can apply Lemma \ref{Lions's lemma}  to $R_2$ directly, thus 
\begin{equation*}
\begin{split}&
\int_0^T\int_{\O}\psi\phi_m(\rho)\bigl[ \Dv(\rho\u\otimes\u)\big]^x_\varepsilon \,dx\,dt
\\&=\big(\int_0^T\int_{\O}\psi\phi_m(\rho)\bigl[\Dv(\rho\u\otimes\u)\bigr]^x_\varepsilon\,dx\,dt
   -\int_0^T\int_{\O}\psi\phi_m(\rho) \Dv(\rho\u\otimes [\u]^x_\varepsilon)\,dx\,dt\big)
\\&+\int_0^T\int_{\O}\psi\phi_m(\rho) \Dv(\rho\u\otimes [\u]^x_\varepsilon)\,dx\,dt
\\&=R_{21}+R_{22}.
\end{split}
\end{equation*}
By Lemma \ref{Lions's lemma},  we have $R_{21}\to 0$ as $\varepsilon\to 0$. The term $R_{22}$ will be calculated in the following way,
\begin{equation*}
\begin{split}
&\int_0^T\int_{\O}\psi\phi_m(\rho) \Dv(\rho\u\otimes [\u]^x_\varepsilon)\,dx\,dt
\\&=\int_0^T\int_{\O}\psi\phi_m(\rho) \Dv(\rho\u) [\u]^x_\varepsilon\,dx\,dt
    +\int_0^T\int_{\O}\psi\phi_m(\rho) \rho\u\cdot \nabla [\u]^x_\varepsilon\,dx\,dt
\\&=\int_0^T\int_{\O}\psi \Dv(\rho\u)[\u_m]^x_\varepsilon\,dx\,dt+\int_0^T\int_{\O}\psi\rho\u\nabla(\phi_m(\rho) [\u]^x_\varepsilon)\,dx\,dt-
\\&\int_0^T\int_{\O}\psi [\u]_\varepsilon^x \cdot\nabla\phi_m(\rho)\rho\u\,dx\,dt
\\&=-\int_0^T\int_{\O}\nabla\psi\rho\u\otimes [\u_m]_\varepsilon^x\,dx\,dt 
     -\int_0^T\int_{\O}\psi\cdot  [\u]_\varepsilon^x \nabla\phi_m(\rho)\rho\u\,dx\,dt,
\end{split}
\end{equation*}
which tends to 
$$-\int_0^T\int_{\O}\nabla\psi\rho\u\otimes \u_m\,dx\,dt -\int_0^T\int_{\O}\psi\cdot \u \nabla\phi_m(\rho)\rho\u\,dx\,dt,$$
as $\varepsilon \to 0$.

\end{proof}
For the other terms in  the momentum equation, we can follow the same way as above method for \eqref{modified  continuity equation} to have 
\begin{equation*}
\label{modified momentum equation}
\begin{split}&\int_0^T\int_{\O}\big(\psi_t\rho\u_m+\nabla\psi\cdot(\rho\u\otimes\u_m 
- 2\phi_m(\rho) (\sqrt{\mu(\rho)}(\Sn+\Sk)
+ \frac{\lambda(\rho)}{2\mu(\rho)} {\rm Tr} (\sqrt{\mu(\rho)} \mathbb S_\mu+ r \mathbb S_r) {\rm Id}     ))
 \\
& + \int_0^T\int_\O \psi(\partial_t\phi_m(\rho)+\u\cdot\nabla\phi_m(\rho))\rho\u
\\&- \int_0^T\int_\O 2 \psi( \sqrt{\mu(\rho)}(\Sn+\Sk) + \frac{\lambda(\rho)}{2\mu(\rho)} {\rm Tr} (\sqrt{\mu(\rho)} \mathbb S_\mu
+ r \mathbb S_r)  {\rm Id}  )\nabla\phi_m(\rho)+\psi\phi_m(\rho) F(\rho,\u)\big)\,dx\,dt
\\&=0.
\end{split}
\end{equation*}
Thanks to \eqref{modified continuity equation}, we have 
\begin{equation}
\label{modified momentum equation}
\begin{split}&\int_0^T\int_{\O}\big(\psi_t\rho\u_m+\nabla\psi\cdot(\rho\u\otimes\u_m
- 2\phi_m(\rho)(\sqrt{\mu(\rho)}(\Sn+ r\Sk)
+ \frac{\lambda(\rho)}{2\mu(\rho)} {\rm Tr} (\sqrt{\mu(\rho)} (\mathbb S_\mu+ r \mathbb S_r) ) {\rm Id} )\\
& - \int_0^T\int_\O \psi \phi'_m(\rho)\frac{\rho}{\sqrt{\mu(\rho)}}\mathrm{Tr} (\Tn)\rho\u
    -\psi\phi_m(\rho) F(\rho,\u)
\\&-\int_0^T\int_\O  2 \psi(\sqrt{\mu(\rho)}(\Sn+ r\Sk)
+ \frac{\lambda(\rho)}{2\mu(\rho)} {\rm Tr} (\sqrt{\mu(\rho)} (\mathbb S_\mu+ r \mathbb S_r) )
   {\rm Id}) \nabla\phi_m(\rho)\big)\,dx\,dt=0.
\end{split}
\end{equation}

\bigskip

The goal of this subsection is to derive the formulation of renormalized solution following the idea in \cite{LaVa}. We choose the function $\bigl[\psi\varphi'([\u_m]_\varepsilon)\bigr]_\varepsilon$ as a test function in \eqref{modified momentum equation}.
As the same argument of Lemma \ref{Lemma for limits-first two terms}, we can show that 
\begin{equation*}
\begin{split}&
\int_0^T\int_{\O}\big(\partial_t\bigl[\psi\varphi'([\u_m]_\varepsilon)\bigr]_\varepsilon\, \rho\u_m
     +\nabla\bigl[\psi\varphi'([\u_m]_\varepsilon)\bigr]_\varepsilon:(\rho\u\otimes\u_m)\big)\,dx\,dt
\\&\to 
\int_0^T\int_{\O}\big(\rho\varphi(\u_m)\psi_t+\rho\u\otimes\varphi(\u_m)\nabla\psi\big)\,dx\,dt,
\end{split}
\end{equation*}
and 
\begin{equation*}
\begin{split}&\int_0^T\int_{\O}\nabla\bigl[\psi\varphi'([\u_m]_\varepsilon)\bigr]_\varepsilon
\big(-2 \phi_m(\rho)(\sqrt{\mu(\rho)}(\Sn+ r\Sk)
+ \frac{\lambda(\rho)}{2\mu(\rho)} {\rm Tr} (\sqrt{\mu(\rho)} \mathbb S_\mu+ r \mathbb S_r) ) {\rm Id} \big) \\
& +\bigl[\psi\varphi'([\u_m]_\varepsilon)\bigr]_\varepsilon  \big(-\phi'_m(\rho)\frac{\rho}{\sqrt{\mu(\rho)}}\mathrm{Tr} (\Tn)\rho\u
\\&-2(\sqrt{\mu(\rho)}(\Sn+ r\Sk)
+ \frac{\lambda(\rho)}{2\mu(\rho)} {\rm Tr} (\sqrt{\mu(\rho)} (\mathbb S_\mu+ r \mathbb S_r) {\rm Id}) )\nabla\phi_m(\rho)+\phi_m(\rho) F(\rho,\u)\big)\,dx\,dt
\\&\to \int_0^T\int_{\O}\nabla(\psi\varphi'(\u_m))\
\big(-2\phi_m(\rho)(\sqrt{\mu(\rho)}(\Sn+ r\Sk)
+ \frac{\lambda(\rho)}{2\mu(\rho)} {\rm Tr} (\sqrt{\mu(\rho)} (\mathbb S_\mu+ r \mathbb S_r) )  {\rm Id} )\big)\\
& +\psi\varphi'(\u_m)
\big(-\phi'_m(\rho)\frac{\rho}{\sqrt{\mu(\rho)}}\mathrm{Tr} (\Tn)\rho\u
\\&-2(\sqrt{\mu(\rho)}(\Sn+ r\Sk)
+ \frac{\lambda(\rho)}{2\mu(\rho)} {\rm Tr} (\sqrt{\mu(\rho)} \mathbb S_\mu+ r \mathbb S_r) )\nabla\phi_m(\rho)+\phi_m(\rho) F(\rho,\u)\big)\,dx\,dt
\end{split}
\end{equation*}
as $\varepsilon$ goes to zero.
Putting these two limits together, we have 
\begin{equation}
\label{weak formulation with m}
\begin{split}&
\int_0^T\int_{\O}\big(\rho\varphi(\u_m)\psi_t+\rho\u\otimes\varphi(\u_m)\nabla\psi\big)
\\&+\nabla\psi\varphi'(\u_m)
\big(-2 \phi_m(\rho)(\sqrt{\mu(\rho)}(\Sn+ r\Sk)
+ \frac{\lambda(\rho)}{2\mu(\rho)} {\rm Tr} (\sqrt{\mu(\rho)} \mathbb S_\mu+ r \mathbb S_r) )\big)\\
& +\psi\varphi''(\u_m)\nabla\u_m\big(-\phi_m(\rho)2(\sqrt{\mu(\rho)}(\Sn+ r\Sk)
+ \frac{\lambda(\rho)}{2\mu(\rho)} {\rm Tr} (\sqrt{\mu(\rho)} \mathbb S_\mu+ r \mathbb S_r) )\big)
\\&+\psi\varphi'(\u_m)
\big(-\phi'_m(\rho)\frac{\rho}{\sqrt{\mu(\rho)}}\mathrm{Tr} (\Tn)\rho\u
-2(\sqrt{\mu(\rho)}(\Sn+ r\Sk)
\\&+ \frac{\lambda(\rho)}{2\mu(\rho)} {\rm Tr} (\sqrt{\mu(\rho)} \mathbb S_\mu+ r \mathbb S_r) )\nabla\phi_m(\rho)+\phi_m(\rho) F(\rho,\u)\big)\,dx\,dt=0.
\end{split}
\end{equation}
Now we should pass to the limit in \eqref{weak formulation with m} as $m$ goes to infinity. To this end, we should keep the following convergences in mind:
\begin{equation}
\begin{split}
\label{basic convergence for m}
&\phi_m(\rho)\;\text{ converges to }1, \quad \text{ for almost every} (t,x)\in \R^+\times\O,\\
&\u_m\text{ converges to } \u, \quad\text{ for almost every}  (t,x)\in \R^+\times\O,\\
&|\rho\phi'_m(\rho)|\leq 2, \quad\text{ and converges to } 0 \text{ for almost every}  (t,x)\in \R^+\times\O.
\end{split}
\end{equation}
We can find that \begin{equation*}
\begin{split}
&\sqrt{\mu(\rho)}\nabla\u_m=\sqrt{\mu(\rho)}\nabla(\phi_m(\rho)\u)
=\phi_m(\rho)\sqrt{\mu(\rho)}\nabla\u+\phi'_m(\rho)\sqrt{\mu(\rho)}\u\cdot\nabla\rho
\\&=\frac{\phi_m(\rho)}{\sqrt{\mu(\rho)}}\big(\nabla(\mu(\rho)\u)-\sqrt{\rho}\u\cdot\frac{\nabla\mu(\rho)}{\sqrt{\rho}}\big)+
\frac{\sqrt{\rho}}{\mu(\rho)^{\frac{3}{4}}}\big(\frac{\sqrt{\mu(\rho)}}{\rho}\mu'(\rho)\nabla\rho\big)
\big(\frac{\rho^{\frac{1}{4}}}{(\mu'(\rho))^{\frac{1}{4}}}\u\big)\big(\phi_m'(\rho)\frac{\mu(\rho)^{\frac{3}{4}}\rho^{\frac{1}{4}}}{(\mu'(\rho))^{\frac{3}{4}}}\big)
\\&=\phi_m(\rho)\Tn
    +\frac{\sqrt{\rho}}{\mu(\rho)^{\frac{3}{4}}}\big(\frac{\sqrt{\mu(\rho)}}{\rho}\mu'(\rho)\nabla\rho\big)
\big(\frac{\rho^{\frac{1}{4}}}{(\mu'(\rho))^{\frac{1}{4}}}\u\big)\big(\phi_m'(\rho)\frac{\mu(\rho)^{\frac{3}{4}}\rho^{\frac{1}{4}}}{(\mu'(\rho))^{\frac{3}{4}}}\big)
\\&=A_{1m}+A_{2m}.
\end{split}
\end{equation*}
Note that 
$$|\phi_m'(\rho)\frac{\mu(\rho)^{\frac{3}{4}}\rho^{\frac{1}{4}}}{(\mu'(\rho))^{\frac{3}{4}}}|\leq C|\phi'_m(\rho)\rho|,$$
 thus $\phi_m'(\rho){\mu(\rho)^{\frac{3}{4}}\rho^{\frac{1}{4}}}/{(\mu(\rho)')^{\frac{3}{4}}}$
  converges to zero for almost every $(t,x).$ 
 Thus, the Dominated convergence theorem yields that  $A_{2m}$ converges to zero as $m\to\infty.$
Meanwhile, the Dominated convergence theorem also gives us $A_{1m}$ converges to $\Tn$ in $L^2_{t,x}$.
Hence, with \eqref{basic convergence for m} at hand, letting $m\to\infty$ in \eqref{weak formulation with m}, one obtains that 
\begin{equation*}
\label{limit for m large}
\begin{split}&
\int_0^T\int_{\O}\big(\rho\varphi(\u)\psi_t+\rho\u\otimes\varphi(\u)\nabla\psi\big)
- 2 \nabla\psi\varphi'(\u)\big((\sqrt{\mu(\rho)}(\Sn+ r\Sk)
\\&+ \frac{\lambda(\rho)}{2\mu(\rho)} {\rm Tr} (\sqrt{\mu(\rho)} (\mathbb S_\mu+ r \mathbb S_r) )
{\rm Id} \big)-2\psi\varphi''(\u)\Tn((\Sn+r \Sk) \\
&+ \frac{\lambda(\rho)}{2\mu(\rho)} {\rm Tr} ((\mathbb S_\mu+ r \mathbb S_r) {\rm Id})
+\psi\varphi'(\u)F(\rho,\u)\big)\,dx\,dt=0.
\end{split}
\end{equation*}
From now,  we  denote $R_{\varphi}=2\psi\varphi''(\u)\Tn((\Sn+ r\Sk)+  \frac{\lambda(\rho)}{2\mu(\rho)} {\rm Tr} ( (\mathbb S_\mu+ r \mathbb S_r) {\rm Id})$.  This ends the proof of Theorem \ref{renorm}.

\end{proof}

\section{renormalized solutions and weak solutions }
The main goal of this section is the proof of Theorem \ref{main result} that obtains the existence of renormalized solutions of the Navier-Stokes equations without the additional terms, thus the existence of weak solutions of the Navier-Stokes equations. 

\subsection{Renormalized solutions}  In this subsection, we will show the existence of renormalized solutions. 
To this end, we need the following lemma of  stability.
\begin{Lemma}
\label{Lemma of stability of renormalized solution}

For any fixed $\alpha_1<\alpha_2$ as in \eqref{mu estimate} and consider sequences $\delta_n$, $r_{0n}$, $r_{1n}$ and $r_{2n}$, 
such that $r_{i,n}\to r_{i}\geq 0$ with $i=0,1,2$ and then $\delta_n\to \delta\geq 0$.
Consider a family of $\mu_n:\R^{+}\to \R^{+}$ verifying 
\eqref{mu estimate} and \eqref{mu estimate1} for the fixed $\alpha_1$ and $\alpha_2$ such that $$\mu_n\to \mu\quad\text{ in }C^0(\R^{+}).$$
Then, if $(\rho_n,\u_n)$ verifies \eqref{priori estimates}-\eqref{priori mu}, up to a subsequence, still denoted $n$, the following convergences hold.
\\
1. The sequence $\rho_n$ convergences strongly to $\rho$ in $C^0(0,T;L^p(\O))$ for any $1\leq p<\gamma.$ \\
2. The sequence $\mu_n(\rho_n)\,  \u_n$ converges to $\mu(\rho)\u$ in $L^{\infty}(0,T;L^p(\O)$ for $p \in [1,3/2)$.
\\
3. The sequence $(\Tn)_n$ convergences to $\Tn$ weakly in $L^2(0,T;L^2(\O))$.\\
4. For every function $H\in W^{2,\infty}(\overline{\R^d})$ and $0<\alpha<{2\gamma}/{\gamma+1}$, we have that $\rho_n^{\alpha} H(\u_n)$ convergences to $\rho^{\alpha}H(\u)$  strongly in $L^p(0,T;\O)$ for $1\leq p<\frac{2\gamma}{(\gamma+1)\alpha}.$ In particular, $\sqrt{\mu(\rho_n)}H(\u_n)$ convergences to $\sqrt{\mu(\rho)}H(\u)$ strongly in $L^{\infty}(0,T;L^2(\O)).$

\end{Lemma}

\begin{proof} 

Using \eqref{priori mu}, the Aubin-Lions lemma gives us, up to a subsequence, $$\mu_n(\rho_n)\to \tilde{\mu}\quad\text{ in }\; C^0(0,T;L^q(\O))$$ for any $q<\frac{3}{2}.$
But $$\sup|\mu_n-\mu|\to 0$$ as $n\to \infty.$ Thus, we have 
\begin{equation}
\label{almost anywhere for mu}
\mu_n(\rho_n)\to \tilde{\mu}(t,x)\quad\text{ in }\; C^0([0,T];L^q(\O)),
\end{equation}
so up to a subsequence, 
\begin{equation*}
\mu(\rho_n)\to \tilde{\mu}(t,x)\;\;\text{a. e}.
\end{equation*}
Note that $\mu$ is increasing function, so it is invertible, and $\mu^{-1}$ is continuous. This implies that $\rho_n\to \rho$ a.e. with $\mu(\rho)=\tilde{\mu}(t,x).$ Together with \eqref{almost anywhere for mu} and $\rho_n$ is uniformly bounded in $L^{\infty}(0,T;L^{\gamma}(\O))$, thus we get part 1.

\bigskip
 
Note that 
$$\nabla\frac{\mu(\rho_n)}{\sqrt{\rho_n}}=\frac{\sqrt{\rho_n}\nabla \mu(\rho_n)}{\rho_n}-\frac{\mu(\rho_n)\nabla\rho_n}{2\rho\sqrt{\rho_n}},$$
thus $$\left|\nabla\frac{\mu(\rho_n)}{\sqrt{\rho_n}}\right|\leq C\left|\sqrt{\rho_n}\right|\left|\frac{\nabla\mu(\rho_n)}{\sqrt{\rho_n}}\right|,$$
so $\nabla\frac{\mu(\rho_n)}{\sqrt{\rho_n}}$ is bounded in $L^{\infty}(0,T;L^2(\O))$, thanks to \eqref{priori estimates}. Using \eqref{priori mu},
we have 
 $\frac{\mu(\rho_n)}{\sqrt{\rho_n}}$ is bounded in $L^{\infty}(0,T;W^{1,2}(\O))$,
 thus it is uniformly bounded in $L^{\infty}(0,T;L^6(\O))$.
 
 On the other hand, $\sqrt{\rho_n}\u_n$ is uniformly bounded in $L^{\infty}(0,T;L^2(\O))$. From Lemma \ref{Compactnesstool1}, we have
 $$\mu(\rho_n)\u_n=\frac{\mu(\rho_n)}{\sqrt{\rho_n}} \sqrt{\rho_n}\u_n\to \mu(\rho)\u\;\;\text{ in }\; L^{\infty}(0,T;L^q(\O))$$ 
 for any $1\leq q<\frac{3}{2}.$
Since $(\Tn)_n$ is bounded in $L^2(0,T;L^2(\O))$, and so, up to a sequence, convergences weakly in $L^2(0,T;L^2(\O))$ to a function $\Tn$.
Using Lemma  \ref{Compactnesstool1}, this gives part~4.
\end{proof}


With Lemma \ref{Lemma of stability of renormalized solution}, we are able to recover the renormalized solutions of Navier-Stokes equations without any additional term by letting $n\to\infty$ in \eqref{limit for m large}.
We state this result in the following Lemma. In this lemma, we fix $\mu$ such that $\varepsilon_1>0$.
\begin{Lemma}
\label{Lemma of existence for ren} For any fixed $\varepsilon_1>0$,
there exists a renormalized solution $(\sqrt{\rho},\sqrt{\rho}\u)$ to the initial value problem \eqref{NS equation}-\eqref{initial data}.
\end{Lemma}

\begin{proof}


We can use Lemma \ref{Lemma of stability of renormalized solution} to pass to the limits for the extra terms. We will have to follow this order: let $r_2$ goes to zero, then $r_1$ tends to zero,  after that  $r_0, \delta, r$ go to zero together.

\noindent -- If $r_2= r_2(n) \to 0$, we just  write 
$$r_2\frac{\rho_n}{\mu'(\rho_n)}|\u_n|^2\u_n=r_2^{\frac{1}{4}}\big(\frac{\rho_n}{\mu'(\rho_n)}\big)^{\frac{1}{4}}\big(\frac{\rho_n}{\mu'(\rho_n)}\big)^{\frac{3}{4}}|\u_n|^2\u_n,$$
and $\mu'(\rho_n)\geq \varepsilon_1 >0,$ so $\big(\frac{\rho_n}{\mu'(\rho_n)}\big)^{\frac{1}{4}}\leq C|\rho_n|^{\frac{1}{4}}$,  thus,
$$r_2\frac{\rho_n}{\mu'(\rho_n)}|\u_n|^2\u_n\to 0 \hbox{ in } L^{\frac{4}{3}}(0,T;L^{\frac{6}{5}}(\O)).$$

\noindent -- For $r_1=r(n)\to 0$, 
$$|r_1\rho_n|\u_n|\u_n|\leq r^{\frac{1}{3}}\rho_n^{\frac{1}{3}}r^{\frac{2}{3}}\rho_n^{\frac{2}{3}}|\u_n|^2,$$
which convergences to zero in $L^{\frac{3}{2}}(0,T;L^{\frac{9}{7}}(\O))$ using the drag term control in the energy and the
information on the pressure law $P(\rho) = a \rho^\gamma$.


\noindent -- For $r_0 = r_0(n) \to 0$, it is easy to conclude that 
$$r_0 \u_{n} \to 0 \hbox{ in } L^2((0,T)\times \Omega).$$

\smallskip

\noindent --  We now consider the limit $r\to 0$ of the term 
$$r\rho_n\nabla\left(\sqrt{K(\rho_n)}\D(\int_0^{\rho_n}\sqrt{K(s)}\,ds)\right).$$
Note the following identity
\begin{equation*}
\label{BCNV relation}
\rho_n\nabla\left(\sqrt{K(\rho_n)}\D(\int_0^{\rho_n}\sqrt{K(s)}\,ds)\right)=
2 \Dv\Bigl(\mu(\rho_n)\nabla^2\bigl(2 {s}(\rho_n)\bigr)\Bigr)
   +\nabla\Bigl(\lambda(\rho_n)\D\bigl(2{s}(\rho_n)\bigr)\Bigr),
\end{equation*}
we only need to  focus on $\Dv\Bigl(\mu(\rho_n)\nabla^2\bigl(2 {s}(\rho_n)\bigr)\Bigr)$ since the same argument holds for   the other term. 
Since  \begin{equation*}
\begin{split}
r\int_{\O}\Dv\Bigl(\mu(\rho_n)&\nabla^2\bigl(2 {s}(\rho_n)\bigr)\Bigr)\psi\,dx
\\&=r\int_{\O}\frac{\rho_n}{\mu_n}\nabla Z(\rho_n) \otimes \nabla Z(\rho_n)\nabla\psi\,dx
+r\int_{\O}\mu_n\nabla{s}(\rho_n)\Delta\psi\,dx\\&=
r\int_{\O}\frac{\rho_n}{\mu_n}\nabla Z(\rho_n) \otimes \nabla Z(\rho_n)\nabla\psi\,dx+r\int_{\O}\sqrt{\mu_n}\nabla Z(\rho_n)\Delta\psi\,dx,
\end{split}
\end{equation*}
the first term can be controlled as  
\begin{equation*}
\begin{split}
&\big|r\int_{\O}\sqrt{\mu_n}\nabla Z(\rho_n)\Delta\psi\,dx\big|\leq Cr^{\frac{1}{2}}\|\sqrt{\mu(\rho_n)}\|_{L^2(0,T;L^2(\O))}\|\sqrt{r}\nabla Z(\rho_n)\|_{L^2(0,T;L^2(\O))}\to 0,
\end{split}
\end{equation*}
thanks to \eqref{J inequality for sequence} and \eqref{priori mu};
and the second term as
\begin{equation*}
\begin{split}
&\big|\int_{\O}\frac{\rho_n}{\mu_n}\nabla Z(\rho_n)\otimes \nabla Z(\rho_n)\nabla\psi\,dx\big|\leq \sqrt{r}\sqrt{r}\int_{\O}\sqrt{\mu(\rho_n)}\frac{\rho_n}{\mu(\rho_n)^{\frac{3}{2}}}|\nabla Z(\rho_n)|^2|\nabla\psi|\,dx
\\&\leq C\|\sqrt{r}\frac{\rho_n}{\mu(\rho_n)^{\frac{3}{2}}}|\nabla Z(\rho_n)|^2\|_{L^2(0,T;L^2(\O))}\|\sqrt{\mu(\rho_n)}\|_{L^2(0,T;L^2(\O))}r^{\frac{1}{2}}\to 0.
\end{split}
\end{equation*}

\smallskip

\noindent -- Concerning the quantity $\delta \rho^{10}$, 
 thanks to  $\mu'_{\varepsilon_1}(\rho)\geq \varepsilon_1>0,$  $\sqrt{\delta}|\nabla\rho^{5}|$ is uniformly bounded in $L^2(0,T;L^2(\O))$. This gives us that $\delta^{\frac{1}{30}}\rho$ is uniformly bounded in $L^{10}(0,T;L^{30}(\O)).$
Thus, we have 
\begin{equation*}\left|
\int_0^T\int_{\O}\delta\rho^{10}\nabla\psi\,dx\,dt\right| \leq C(\psi) \delta^{\frac{2}{3}}\|\delta^{\frac{1}{3}}\rho^{10}\|_{L^1(0,T;L^3(\O))}\to 0
\end{equation*}
as $\delta\to0.$

With Lemma \ref{Lemma of stability of renormalized solution} at hand, we are ready to recover the renormalized solutions to  \eqref{NS equation}-\eqref{initial data}. By part 1 and part 2 of Lemma \ref{Lemma of stability of renormalized solution},
we are able to pass to the limits on the continuity equation. 
Thanks to part 4 of Lemma \ref{Lemma of stability of renormalized solution}, 
$$\sqrt{\mu(\rho_n)}\varphi'(\u_n)\to \sqrt{\mu(\rho)}\varphi'(\u) \quad\text{ in }\;\; L^{\infty}(0,T;L^2(\O)).$$ 
 With the help of Lemma \ref{compactuniforme},  we can pass to the limit on pressure, thus
 we can recover the  renormalized solutions.

\end{proof}

\subsection{Recover weak solutions from renormalized solutions}
In this part, we can recover the weak solutions from the renormalized solutions constructed in Lemma \ref{Lemma of existence for ren}. Now we show that Lemma \ref{Lemma of existence for ren} 
is valid without the condition $\varepsilon_1>0$. For such a $\mu$, we construct a sequence $\mu_n$ converging to $\mu$ in $C^0(\R^+)$ and such that $\varepsilon_{1n}=\inf \mu_n'>0$. Lemma \ref{Lemma of stability of renormalized solution} shows that, up to a subsequence, $$\rho_n\to\rho\;\;\text{ in }\; C^0(0,T;L^p(\O))$$
and $$\rho_n\u_n\to\rho\u\;\;\text{ in } L^{\infty}(0,T;L^{\frac{p+1}{2p}}(\O))$$ for any $1\leq p<\gamma,$
where $(\rho,\sqrt{\rho}\u)$ is a renormalized solution to \eqref{NS equation}.

Now, we want to show that this renormalized solution is also a weak solution  in the sense of Definition 1.2. To this end, we introduce a non-negative smooth function $\Phi:\R\to\R$ such that it has a compact support and $\Phi(s)=1$ for any $-1\leq s\leq1.$  Let $\tilde{\Phi}(s)=\int_0^s\Phi(r)\,dr$, we define 
$$\varphi_n(y)=n\tilde{\Phi}(\frac{y_1}{n})\Phi(\frac{y_2}{n})....\Phi(\frac{y_N}{n})$$
for any $y=(y_1,y_2,....,y_N)\in \R^N$.

Note that $\varphi_n$ is bounded in $W^{2,\infty}(\R^N)$ for any fixed $n>0$, $\varphi_n(y)$ converges everywhere to $y_1$ as $n$ goes to infinity, $\varphi_n'$ is uniformly bounded in $n$ and converges everywhere to unit vector $(1,0,....0)$, and $$\|\varphi_n''\|_{L^{\infty}}\leq \frac{C}{n}\to 0$$
as $n$ goes to infinity. This allows us to control the measures in Definition \ref{def_renormalise_u} as follows
  $$
 \|R_{\vfi_n}\|_{ \mathcal{M}(\R^+\times\O)}+ \|\overline{R}^1_{\vfi_n}\|_{ \mathcal{M}(\R^+\times\O)}
 +  \|\overline{R}^2_{\vfi_n}\|_{ \mathcal{M}(\R^+\times\O)} \leq C \|\vfi''_n\|_{L^\infty(\R)}\to 0
 $$
as $n$ goes to infinity. Using this function $\varphi_n$ in the equation of Definition \ref{def_renormalise_u}, the Lebesgue's Theorem gives us the equation on $\rho\u_1$ in Definition 1.2 by passing limits as $n$ goes to infinity. In this way, we are able to get full vector equation on $\rho\u$
by permuting the directions. 
Applying the Lebesgue's dominated convergence Theorem, one obtains \eqref{Smu} by passing to limit in \eqref{eq_viscous_renormaliseAAA} with $i=1$ and the function $\varphi_n$. Thus, we have shown that the renormalized solution is also a weak solution.

 \section{acknowledgement} Didier Bresch is supported by the SingFlows project, grant ANR-18-CE40-0027 and by the project Bords, grant ANR-16-CE40-0027 of the French National Research Agency (ANR). He want to thank Mikael de la Salle (CNRS-UMPA Lyon) for his efficiency
 within the National Committee for Scientific Research - CNRS (section 41)  which allowed him to visit the university of Texas at Austin in January 2019 with precious progress on this work during this period.
 Alexis Vasseur is partially supported by the NSF grant: DMS 1614918.  Cheng Yu is partially supported by the start-up funding of University of Florida.


\begin{thebibliography}{99}

\bibitem{AnSp0}
\sc  P. Antonelli, S. Spirito.
\rm Global existence of weak solutions to the Navier-Stokes-Korteweg equations.
\rm ArXiv:1903.02441 (2019).

\bibitem{AnSp}
\sc  P. Antonelli, S. Spirito.
\rm On the compactness of weak solutions to the Navier-Stokes-Korteweg 
equations for capillary fluids. 
\rm ArXiv:1808.03495 (2018).

\bibitem{AnSp1}
\sc P. Antonelli, S. Spirito.
\rm Global Existence of Finite Energy Weak Solutions of Quantum Navier-Stokes Equations. 
\it Archive of Rational Mechanics and Analysis, \rm  225 (2017), no. 3, 1161--1199.

\bibitem{AnSp2}
\sc P. Antonelli, S. Spirito.
\rm On the compactness of finite energy weak solutions to the Quantum Navier-Stokes equations.
 \it J. of Hyperbolic Differential Equations, \rm 15 (2018), no. 1, 133--147.

\bibitem{BP}
\sc C. Bernardi, O. Pironneau.
\rm On the shallow water equations at low Reynolds number.
\it Comm. Partial Differential Equations \rm 16 (1991), no. 1, 59--104.

\bibitem{BD}
\sc D. Bresch, B. Desjardins.
\rm Existence of global weak solutions for 2D viscous shallow water equations and convergence to the quasi-geostrophic model.
\it Comm. Math. Phys.,  \rm 238 (2003), no.1-3, 211--223.

\bibitem{BCNV}
\sc D. Bresch, F. Couderc, P. Noble, J.-P. Vila.
\rm A generalization of the quantum Bohm identity: Hyperbolic CFL
condition for Euler--Korteweg equations.
 \it C.R. Acad. Sciences Paris \rm  Volume 354, Issue 1, 39--43, (2016).


 \bibitem{BD2006}
 \sc D. Bresch and B. Desjardins.
 \rm On the construction of approximate solutions for the 2D viscous shallow water model and for compressible Navier-Stokes models.
 \it J. Math. Pures Appl. \rm (9) 86 (2006), no. 4, 362--368.

\bibitem{BrDeFormula}
\sc D. Bresch, B. Desjardins.
\rm Quelques mod\`eles diffusifs capillaires de type Korteweg.
\it C. R. Acad. Sci. Paris, \rm section m\'ecanique,
\rm {\bf 332}, no. 11, 881--886, (2004).

\bibitem{BrDeSpringer}
\sc D. Bresch, B. Desjardins.
\rm Weak solutions via the total energy formulation and their quantitative properties - density dependent viscosities.
\rm In: Y. Giga, A. Novotn\'y (\'eds.) Handbook of Mathematical Analysis in Mechanics of Viscous Fluids. Springer, Berlin (2017).

\bibitem{BDL}
\sc D. Bresch, B. Desjardins, Chi-Kun Lin.
\rm On some compressible fluid models: Korteweg, lubrication, and shallow water systems.
\it Comm. Partial Differential Equations \rm 28 (2003), no. 3-4, 843--868.


\bibitem{BDZ}
\sc D. Bresch, B. Desjardins, E. Zatorska.
\rm  Two-velocity hydrodynamics in Fluid Mechanics, Part II.
 Existence of global $\kappa$-entropy solutions to compressible Navier-Stokes system with degenerate viscosities.  \it J. Math. Pures Appl. \rm Volume 104, Issue 4, 801--836  (2015).

\bibitem{BJ}\sc D.
 Bresch, P.-E. Jabin.
 \rm Global existence of weak solutions for compressible Navier-Stokes equations: thermodynamically unstable pressure and anisotropic viscous stress tensor. \it Ann. of Math. \rm (2) 188 (2018), no. 2, 577-684.
 
 \bibitem{BrGiLa}
 \sc D. Bresch, I. Lacroix-Violet, M. Gisclon.
 \rm On Navier-Stokes-Korteweg and Euler-Korteweg systems: Application to quantum fluids models.
 \rm To appear in \it Arch. Rational Mech. Anal. \rm (2019).
 
 \bibitem{BrMuZa}
 \sc D. Bresch, P. Mucha, E. Zatorska.
 \rm Finite-energy solutions for compressible two-fluid Stokes system.
 \it Arch. Rational Mech. Anal., \rm 232, Issue 2, (2019), 987--1029.
 
 \bibitem{BuHa}
 \sc C. Burtea, B. Haspot.
 \rm New effective pressure and existence of global strong solution for compressible 
 Navier-Stokes equations with general viscosity coefficient in one dimension.
 \rm arXiv:1902.02043 (2019).
 
 \bibitem{CaCaHi}
 \sc R. Carles, K. Carrapatoso, M. Hillairet.
 \rm Rigidity results in generalized isothermal fluids.
 \it Annales Henri Lebesgue, \rm 1, (2018), 47--85. 
 

\bibitem{CoDrNgPa} 
\sc P. Constantin, T. Drivas, H.Q. Nguyen, F. Pasqualotto.
\rm Compressible fluids and active potentials.
\rm ArXiv:1803.04492.


 \bibitem{DNV}
 \sc B. Ducomet, S. Necasova, A. Vasseur,
 \rm On spherically symmetric motions of a viscous compressible barotropic and self-graviting gas.
 \it J. Math. Fluid Mech. 
 \rm 13 (2011), no. 2, 191--211.


\bibitem{FNP}
\sc E. Feireisl, A. Novotn\'{y}, H. Petzeltov\'{a}.
\rm On the existence of globally defined weak solutions to the Navier-Stokes
equations.
\it J. Math. Fluid Mech. \rm \textbf{3} (2001), 358--392.


\bibitem{F04}
\sc E. Feireisl.
\rm Compressible Navier--Stokes Equations with a Non-Monotone Pressure Law.
\it J. Diff. Eqs \rm 183, no 1, 97--108, (2002).

\bibitem{GJX}
\sc Z. Guo, Q. Jiu, Z. Xin.
\rm Spherically symmetric isentropic compressible flows with density-dependent viscosity coefficients.
\it SIAM J. Math. Anal. \rm 39 (2008), no. 5, 1402--1427. 



\bibitem{Haspot}
\sc B. Haspot.
\rm  Existence of global strong solution for the compressible Navier-Stokes equations with degenerate viscosity coefficients in 1D. \it  Mathematische Nachrichten, \rm  291 (14-15), 2188--2203, (2018).


\bibitem{Hoff87}
\sc D. Hoff.
\rm Global existence for 1D, compressible, isentropic Navier-Stokes equations with large initial data. 
\it Trans. Amer. Math. Soc. \rm 303 (1987), no. 1, 169--181.









\bibitem{JXZ}
\sc S. Jiang, Z. Xin, P. Zhang. \rm
Global weak solutions to 1D compressible isentropic Navier-Stokes equations with density-dependent viscosity. \it
Methods Appl. Anal. \rm 12 (2005), no. 3, 239--251. 

\bibitem{JZ}
\sc S. Jiang, P. Zhang.
\rm On spherically symmetric solutions of the compressible isentropic Navier-Stokes equations. \it Comm. Math. Phys. \rm 215 (2001), no. 3, 559--581. 


\bibitem{J}
\sc A. J\"ungel.
\rm Global weak solutions to compressible Navier-Stokes equations for quantum fluids.
\it SIAM J. Math. Anal. \rm 42 (2010), no. 3, 1025--1045.

\bibitem{JuMa}
\sc A. J\"ungel, D. Matthes.
\rm The Derrida--Lebowitz-Speer-Spohn equations: Existence, uniqueness, and Decay rates of the
solutions.
\it SIAM J. Math. Anal.,
\rm 39(6), (2008), 1996--2015.


\bibitem{KS}
\sc A.V. Kazhikhov, V.V. Shelukhin.
\rm Unique global solution with respect to time of initial-boundary value problems for one-dimensional equations of a viscous gas. \it
J. Appl. Math. Mech. \rm
41 (1977), no. 2, 273--282.; translated from \it Prikl. Mat. Meh.\rm  41 (1977), no. 2, 282--291(Russian).


\bibitem{Kan}
\sc J. I. Kanel. \rm A model system of equations for the one-dimensional motion of a gas.
\it Differ. Uravn. \rm 4 (1968), 721--734 (in Russian).

\bibitem{LaVa}
\sc I. Lacroix-Violet, A. Vasseur.
\rm Global weak solutions to the compressible quantum Navier-Stokes 
       equation and its semi-classical limit.
\it J. Math. Pures Appl. \rm (9) 114 (2018), 191--210.

\bibitem{Le} 
\sc J. Leray.
\rm  Sur le mouvement d'un fluide visqueux remplissant l'espace,
\it  Acta Math. \rm 63 (1934), 193--248.

\bibitem{LiLiXi}
\sc H.L. Li, J. Li, Z.P. Xin.
\rm Vanishing of vacuum states and blow-up phenomena of the compressible Navier--Stokes equations.
\it Comm. Math. Phys., \rm 281, 401--444 (2008).

\bibitem{LiXi}
\sc J. Li, Z.P. Xin.
\rm  Global Existence of Weak Solutions to the Barotropic Compressible Navier-Stokes Flows with Degenerate Viscosities.
\rm  arXiv:1504.06826 (2015).







\bibitem{Lions}
\sc P.-L. Lions.
\it Mathematical topics in fluid mechanics.  Vol. 2. Compressible models.
\rm  Oxford Lecture Series in Mathematics and its Applications, 10. Oxford Science Publications. The Clarendon Press, Oxford University Press, New York, 1998.





\bibitem{MaMiMuNoPoZa}
\sc D. Maltese, M. Michalek, P. Mucha, A. Novotny, M. Pokorny, E. Zatorska.
\rm Existence of weak solutions for compressible Navier-Stokes with entropy transport.
\it J. Differential Equations, \rm 261, No. 8, 4448--4485 (2016)

\bibitem{MV}
\sc A. Mellet, A. Vasseur.
\rm On the barotropic compressible Navier-Stokes equations.
\it Comm. Partial Differential Equations \rm 32 (2007), no. 1-3, 431--452.


\bibitem{MV2}
\sc A. Mellet, A. Vasseur.
\rm Existence and uniqueness of global strong solutions for one-dimensional compressible Navier-Stokes equations.
.\it  SIAM J. Math. Anal. \rm  39 (2007/08), no. 4, 1344--1365.

\bibitem{MuPoZa}
\sc P.B. Mucha, M. Pokorny, E. Zatorska.
\rm Approximate solutions to a model of two-component reactive flow.
\it Discrete Contin. Dyn. Syst. Ser. \rm S, 7, No. 5 , 1079--1099 (2014).

\bibitem{No}
\sc A. Novotny.
\rm Weak solutions for a bi-fluid model of a mixture ot two compressible non interacting fluids.
\rm Submitted (2018).

\bibitem{NoPo}
\sc A. Novotny, M. Pokorny.
\rm Weak solutions for some compressible multi-component fluid models.
\rm Submitted (2018).

\bibitem{PW}
\sc P.I. Plotnikov, W.  Weigant. \rm 
Isothermal Navier-Stokes equations and Radon transform. \it
SIAM J. Math. Anal.  \rm47 (2015), no. 1, 626--653. 

\bibitem{Ro}
\sc F. Rousset.
\rm Solutions faibles de l'\'equation de Navier-Stokes des fluides compressible [d'apr\`es A. Vasseur et C. Yu].
\rm S\'eminaire Bourbaki, 69\`eme ann\'ee, 2016--2017, no 1135.

\bibitem{S}
\sc D. Serre.
\rm Solutions faibles globales des \'equations de Navier-Stokes pour un fluide compressible.
\it  C. R. Acad. Sci. Paris. \rm I Math. 303 (1986), no. 13, 639--642.


\bibitem{VK}
\sc V. A. Vaigant, A. V. Kazhikhov.
\rm On the existence of global solutions of two-dimensional Navier-Stokes equations of a compressible viscous fluid. (Russian).
\it  Sibirsk. Mat. Zh. \rm 36 (1995), no. 6, 1283-1316, ii; 
translation in \it  Siberian Math. J.
\rm 36 (1995), no.6, 1108--1141.


\bibitem{VY-1}
\sc A. Vasseur, C. Yu.
\rm Global weak solutions to compressible quantum Navier-Stokes equations with damping.
 \it  SIAM J. Math. Anal. \rm 48 (2016), no. 2, 1489--1511.


\bibitem{VY}
\sc A. Vasseur, C. Yu.
\rm Existence of Global Weak Solutions for 3D Degenerate Compressible Navier-Stokes Equations.
\it Inventiones mathematicae \rm (2016), 1--40.

\bibitem{VWY}
\sc A. Vasseur, H. Wen, C. Yu.
\rm Global weak solution to the viscous two-phase model with  finite energy.
\it To appear in J. Math Pures Appl. \rm (2018).

\bibitem{Z}
\sc E. Zatorska.
\rm On the flow of chemically reacting gaseous mixture.
\it J. Diff. Equations. \rm 253 (2012) 3471--3500.

\end{thebibliography}
\end{document}